\documentclass[11pt]{amsart}
\usepackage{amsfonts}
\usepackage{amsmath,amssymb,amsthm}\allowdisplaybreaks[4]
\usepackage{url}
\usepackage{graphicx}   
\usepackage{verbatim}       
\usepackage{layout}
\usepackage{galois}
\usepackage{mathrsfs}
\usepackage{bbm}
\usepackage{color}
\usepackage{enumitem}
\setlist[enumerate,1]{label=(\Roman*)}
\usepackage{cite}
\numberwithin{equation}{section}
\usepackage[linkcolor=red,citecolor=blue]{hyperref}

\newtheorem{theorem}{Theorem}[section]
\newtheorem{lemma}[theorem]{Lemma}
\newtheorem{proposition}[theorem]{Proposition}
\newtheorem{remark}[theorem]{Remark}

\newtheorem{corollary}[theorem]{Corollary}

\newtheorem{example}[theorem]{Example}
\newtheorem{condition}[theorem]{Condition}

\def\[{\left[}\def\]{\right]}
\def\<{\left<}\def\>{\right>}
\def\({\left(}\def\){\right)}

 \def\beqlb{\begin{eqnarray}}\def\eeqlb{\end{eqnarray}}
 \def\beqnn{\begin{eqnarray*}}\def\eeqnn{\end{eqnarray*}}
 
\def\ar{\!\!&}

\def\al{{\alpha}}\def\be{{\beta}}
\def\ep{{\epsilon}}\def\ga{{\gamma}}
\def\la{{\lambda}}

\def\La{{\Lambda}}

\begin{document}
\title{Exact modulus of continuities for  $\Lambda$-Fleming-Viot processes with Brownian spatial motion}
\author{Huili Liu and Xiaowen Zhou}
\address{Huili Liu: School of Mathematical Sciences, Hebei Normal University,
Shijiazhuang, Hebei, China}
\email{liuhuili@hebtu.edu.cn}
\address{Xiaowen Zhou: Department of Mathematics and Statistics, Concordia University,
1455 de Maisonneuve Blvd. West, Montreal, Quebec, H3G 1M8, Canada}
\email{xiaowen.zhou@concordia.ca }
\date{\today}

\begin{abstract}
For a class of $\Lambda$-Fleming-Viot processes with Brownian spatial  motion in $\mathbb{R}^d$ whose associated $\Lambda$-coalescents come down from infinity, we obtain  sharp global and local  modulus of continuities for the ancestry processes recovered from the lookdown representations. As applications, we prove both global and local modulus of continuities for the $\Lambda$-Fleming-Viot support processes. In particular, if the $\Lambda$-coalescent is the Beta$(2-\beta,\beta)$ coalescent for $\beta\in(1,2]$  with $\beta=2$ corresponding to Kingman's coalescent, then for  $h(t)=\sqrt{t\log (1/t)}$, the global modulus of continuity holds for the support process with modulus function $\sqrt{2\beta/(\beta-1)}h(t)$, and both the left and right local modulus of continuities hold for the support process with modulus function $\sqrt{2/(\beta-1)}h(t)$.

\end{abstract}
\maketitle \pagestyle{myheadings} \markboth{\textsc{Modulus of continuities for $\Lambda$-Fleming-Viot processes}} {\textsc{H. Liu and X. Zhou}}
\section{Introduction}
Fleming-Viot processes are probability-measure-valued Markov processes that can arise as scaling limits of well known population genetic models such as the generalized Cannings model and  Moran model that describe the evolution of relative frequencies for different types of alleles in a large population. The associated discrete population may undergo resampling (reproduction) together with possible mutation, selection and
recombination; see Ethier and Kurtz \cite{EtKu93} and Etheridge \cite{Eth12} for surveys on Fleming-Viot process and related  population genetic models. We also refer to Birkner and Blath \cite{BB} for a nice illustration and review on relations between the measure-valued process and the approximating discrete state population genetic models.

The mutation in a Fleming-Viot process  can  also  be interpreted as the underlying spatial motion.
The resampling mechanism is associated to a coalescent. In particular, the Kingman-Fleming-Viot process is dual to Kingman's coalescent with binary collisions. The $\La$-Fleming-Viot process is dual to the more general $\La$-coalescent with multiple collisions, which allows a more general reproduction mechanism during the resampling event for the corresponding discrete model. More precisely, the $\La$-Fleming-Viot process corresponds to a discrete population in which an individual can occasionally give birth to a large number of children that is comparable to the population size. In this paper, we are interested in a selection neutral $\Lambda$-Fleming-Viot process that only involves resampling and $d$-dimensional Brownian spatial motion.

The Fleming-Viot process is an important example of the so called measure-valued superprocess. Another important example of the superprocess is  the Dawson-Watanabe superprocess that describes the evolution of a spatially distributed continuous-state branching process and arises as
the high density limit of empirical measures for branching Brownian motions.  We refer to Dawson \cite{Dawson93}, Perkins \cite{perkins} and Li \cite{Li} for comprehensive reviews on Dawson-Watanabe superprocesses and related topics. In particular, the Fleming-Viot process can be identified as a time-changed and renormalized Dawson-Watanabe superprocess conditioning on nonextinction; see  \cite[Chapter 4]{Eth12} for detailed discussions.

The support properties of  Dawson-Watanabe superprocesses have been well studied. For example, Dawson et al. \cite{DIP} investigated the sample path properties of super-Brownian motion, a superprocess with Brownian spatial motion and
binary branching mechanism, and proved a sharp modulus of continuity result for the associated historical process, which leads to a uniform one-sided modulus of continuity for the  support process. Also see Perkins \cite[Section III.1]{perkins} for a proof of modulus of continuity for the  support of super-Brownian motion without using nonstandard analysis.
Exact Hausdorff functions are also found for  the range and the closed support. Dawson and Vinogradov \cite{Da94} generalized the result in \cite{DIP} and  obtained both global and local (at fixed time)  modulus of continuities for the closed support of  superprocess with Brownian spatial motion and stable branching mechanism.
We also refer to Dawson \cite{Dawson93} and references therein for a collection of these results. Proofs of the above results  involve the branching particle system approximation, the historical cluster representation, the Palm distribution for the canonical measure,  estimates obtained from PDE associated with the Laplace functional for the superprocess and nonstandard analysis techniques. In particular, when the initial measure is the  Dirac mass at the origin, using nonstandard analysis Tribe \cite{Tri89} established the exact right modulus of continuity at time $0$ for the support of super-Brownian motion.  This was generalized in \cite[Theorem 3.1]{Da94} to a similar result for  superprocess with stable branching. In another generalization of Tribe's result, using the Brownian snake representation, Dhersin and Le Gall \cite{DL98} proved a Kolmogorov's test on the local modulus of continuity at time $0$.

Although the Fleming-Viot processes are closely related to the Dawson-Watanabe superprocesses,
due to the lack of infinite divisibility, many methods for the study of  Dawson-Watanabe superprocesses are not available for the treatment of Fleming-Viot processes and consequently, there are many fewer results concerning the path properties for Fleming-Viot processes.
Some  early results on the compact support property and the Hausdorff dimension of the support for the Kingman-Fleming-Viot process were shown by Dawson and Hochberg \cite{Dawson} and by Reimers \cite{Rei}.

To keep track of the genealogy for the Fleming-Viot processes, Donnelly and Kurtz \cite{DK96,DK99a,DK99b} proposed several versions of lookdown representation
to the Kingman and the generalized Fleming-Viot processes. These representations turned out to be very useful to the study of Fleming-Viot processes. More precisely, the lookdown representation offers  a discrete representation which involves an  exchangeable system of countably many coupled particles driven by a Poisson point process such that the empirical measures of the particles converge to the corresponding measure-valued process. Looking backwards in time, an ancestry process can be recovered from the lookdown representation to keep track of genealogy of the particles. The lookdown representation also implies the duality between the Fleming-Viot process and its associated coalescent. More details are provided in Section \ref{sec:lookdown} below. The lookdown representation for the $\Lambda$-Fleming-Viot process is extended  in Birkner et al. \cite{MJM} to the more general $\Xi$-Fleming-Viot process that involves simultaneous multiple collisions. Another elegant representation of the mutationless $\Lambda$-Fleming-Viot process using a flow of bridges can be found in Bertoin and Le Gall \cite{BerLeGa03}. We refer to Etheridge and  Kurtz \cite{EK19} for more recent work on particle representations of population models such as the spatial $\Lambda$-Fleming-Viot process.

Birkner and Blath \cite{BB} pointed out that the $\Lambda$-Fleming-Viot process does not have compact
support if the associated $\Lambda$-coalescent does not come down from infinity. Applying the lookdown representation,  for a class of $\Lambda$-Fleming-Viot processes that come down from infinity  Liu and Zhou \cite{LZ1,LZ2} showed  the compact support property and a one-sided modulus of continuity for the support  at each fixed time, and found the bounds of Hausdorff dimensions on the support and range. Note that the compact support property fails if the spatial motion allows jumps. In particular, for a class of $\Lambda$-Fleming-Viot processes with L\'evy spatial motion whose associated $\Lambda$-coalescents come down from infinity, Hughes and Zhou \cite{HZ22}  proved that instantaneous support propagation occurs in the sense that starting with any initial measure, the Fleming-Viot support can reach arbitrarily far away at any positive time. In all these works the lookdown representation serves as an efficient substitute for the historical process of superprocess.

In this paper we are  interested in improving the existing modulus of continuity in \cite{LZ2} for the $\Lambda$-Fleming-Viot process. Given $h(x)=\sqrt{x\log\(1/x\)}$ for  $x>0$, it is well known that $\sqrt{2}h(x)$ is the exact modulus of continuity function for almost all Brownian paths. Dawson et al. \cite{DIP} showed that ${2}h(x)$ is the exact modulus of continuity of super-Brownian motion. A natural question is whether there exists an exact modulus of continuity for the Fleming-Viot process.

The goal of this paper is to establish sharp global and local modulus of continuities for the ancestry processes recovered from the lookdown  representation of a class of $\Lambda$-Fleming-Viot processes with Brownian spatial motion and with the associated $\Lambda$-coalescents coming down from infinity, which lead to some improved modulus of continuities for the Fleming-Viot support processes.
To show the upper bounds for the global modulus of continuity for the ancestry process, we modify the approach in \cite{LZ2} that relies on the lookdown representation and adopt the scheme of partitioning the time interval in \cite{DIP} that is an adaption of Paul L\'evy's idea for showing modulus of continuity for Brownian motion. A more precise lower bound for the number of ancestors in the ancestry process  is found by applying an estimate on the dual $\Lambda$-coalescent in  Berestycki et al. \cite{BBL}, which
facilitates a derivation of the lower bounds for the global modulus of continuity.
The coincidence of the upper and lower bounds results in an exact global modulus of continuity.
Adjusting this approach we also obtain sharp local left and right modulus of continuities for the ancestry process, respectively. Those sufficient conditions on the $\Lambda$-coalescent are rather general.

To complete the introduction, we provide an outline for the rest of the paper as follows. In Section \ref{sec:coa}, we introduce
the $\La$-coalescent, its coming down from infinity property and some conditions describing its behaviors.
In Section \ref{sec:lookdown}, we briefly discuss the lookdown representation for the $\La$-Fleming-Viot process with Brownian spatial motion, and then recover the associated $\Lambda$-coalescent and ancestry process from the lookdown representation. In Section \ref{sec:theorems}, we present the main results concerning the global and local modulus of continuities for the ancestry process and the support process. Section \ref{sec:result} contains  some estimations, proofs and applications.


\section{ $\Lambda$-coalescent}\label{sec:coa}
In this section, we first introduce the $\Lambda$-coalescent and then discuss its coming down from infinity property. Finally, some  conditions are presented to describe the behaviors of the coalescent. 

\subsection{The $\Lambda$-coalescent}
A systematic introduction on exchangeable coalescents can be found in Bertoin \cite{Bertoin}. Let $[n]\equiv\{1,\ldots,n\}$ and $[\infty]\equiv\{1,2,\ldots\}$. An ordered partition of $D\subset[\infty]$ is a countable collection $\pi=\{\pi_i, i=1,2,\ldots\}$ of disjoint blocks such that $\cup_{i}\pi_i=D$ and $\min\pi_i<\min\pi_j$ for $i<j$. Then blocks in $\pi$ are ordered by their least elements. Denote by $\mathcal{P}_n$ the set of ordered partitions of $[n]$ and by $\mathcal{P}_\infty$  the set of ordered partitions of $[\infty]$. Write $\mathbf{0}_{[n]}\equiv\left\{\{1\},\ldots,\{n\}\right\}$ for the partition of $[n]$ consisting of singletons and $\mathbf{0}_{[\infty]}$ for the partition of $[\infty]$ consisting of singletons. Given $n\in[\infty]$ and $\pi\in\mathcal{P}_\infty$, let $R_n(\pi)\in\mathcal{P}_n$ be the restriction of $\pi$ to $[n]$.

Kingman's coalescent is a $\mathcal{P}_{\infty}$-valued time homogeneous Markov process such that all different pairs of blocks independently merge at the same rate. The $\La$-coalescent generalizes Kingman's coalescent by allowing multiple collisions (see Pitman \cite{Pitman99} and Sagitov \cite{Sag99}). It is a $\mathcal{P}_{\infty}$-valued Markov process $\Pi\equiv\(\Pi(t)\)_{t\geq 0}$ such that for each $n\in[\infty]$, its restriction to $[n]$, $\Pi_n\equiv\(\Pi_n(t)\)_{t\geq 0}$ is a $\mathcal{P}_n$-valued Markov process with transition rates below: given there are $b$ blocks in the partition, each $k$-tuple of them ($2\leq k\leq b$) independently merges to form a single block at rate
$\lambda_{b,k}=\int_{[0,1]}x^{k-2}(1-x)^{b-k}\Lambda(dx),$ where $\Lambda$ is a finite measure on $[0,1]$. For $n=2,3,\ldots$, denote by
$\lambda_n=\sum_{k=2}^n{n\choose k}\lambda_{n,k}$ the total coalescence rate starting with $n$ blocks. In addition, define $\gamma_n=\sum_{k=2}^n\(k-1\){n\choose k}\la_{n,k}$ as the rate at which the number of blocks decreases.

Provided $\La(\{0\})=0$, Pitman \cite{Pitman99} introduced  a Poisson point construction for the associated $\La$-coalescent. Let $\Psi$
be a Poisson point process on $(0,\infty)\times\{0,1\}^{\infty}$ with intensity measure $dtL(d\xi)$, where
\beqlb\label{eq:L}
L(d\xi)\equiv\int_{(0,1]} x^{-2}\La(dx)Q_{x}(d\xi)
\eeqlb
and under $Q_x$,  $\xi=\(\xi_1,\xi_2,\xi_3,\ldots\)$ is a sequence of independent Bernoulli trials with $Q_x\(\xi_i=1\)=x$ for all $i$.
Given an arbitrary partition $\pi\in\mathcal{P}_{\infty}$, let $\Pi_n(0)$ be the restriction of $\pi$ to $[n]$.
At a point $(t,\xi)$ of $\Psi(\cdot)$ with $\sum_{i}\xi_i\geq 2$,
$\Pi_n(t)$ is obtained from $\Pi_n(t-)$ by merging those blocks whose indexes belong to the collection $\{i:\xi_i=1\}$, while all other blocks with indexes in $\{i:\xi_i=0\}$ remain unchanged.
By such a construction, $\Pi_n$ is the
restriction to $[n]$ of $\Pi_{n+1}$ for every $n$. The Poisson point process $\Psi(\cdot)$ therefore
determines a unique $\mathcal{P}_{\infty}$-valued $\La$-coalescent process $\Pi$ whose restriction to
$[n]$ is $\Pi_n$ for every $n$.

\subsection{Coming down from infinity}
Given any $\La$-coalescent $\Pi=\(\Pi(t)\)_{t\geq 0}$ with $\Pi(0)=\mathbf{0}_{[\infty]}$, let $N(t)$ be the number of blocks in the partition $\Pi(t)$. The $\La$-coalescent $\Pi$ comes down from infinity if $\mathbb{P}\(N(t)<\infty\)=1$ for all $t>0$ and it  stays infinite if $\mathbb{P}\(N(t)=\infty\)=1$ for all $t>0$. Given $\Lambda(\{1\})=0$, a $\La$-coalescent either comes down from infinity with probability one or stays infinite with probability one; see Pitman \cite{Pitman99}. Throughout this paper, we always assume that $\Lambda(\{1\})=0$. Schweinsberg \cite{Jason2000} showed that
\begin{itemize}
\item the $\La$-coalescent comes down from infinity if and only if
$\sum_{n=2}^{\infty}\gamma_n^{-1}<\infty$;
\item the
$\La$-coalescent stays infinite if and only if
$\sum_{n=2}^{\infty}\gamma_n^{-1}=\infty$.
\end{itemize}

What follows are two examples of coalescents.
\begin{example}
If $\La=\delta_0$, a point mass at $0$, the associated coalescent is Kingman's
coalescent and comes down from infinity.
\end{example}

\begin{example}
If $\Lambda$ is the Beta$(2-\be,\be)$ function such that
\beqlb\label{eq:beta1}
\La(dx)=\frac{\Gamma(2)}{\Gamma(2-\be)\Gamma(\be)}x^{1-\be}\(1-x\)^{\be-1}dx
\eeqlb
for $\be\in(0,2)$, the associated $\Lambda$-coalescent is a Beta$(2-\be,\be)$ coalescent.
\begin{itemize}
\item In case of $\be\in\left(0,1\right]$, the Beta$(2-\be,\be)$-coalescent stays infinite.
\item In case of $\be\in\left(1,2\right)$, the Beta$(2-\be,\be)$-coalescent comes down from infinity.
\end{itemize}
\end{example}
For convenience we often identify Kingman's coalescent as the Beta$(0,2)$-coalescent.

\subsection{Some conditions for the main results}
Intuitively, if the coalescence rate is larger, then the speed of coming down is faster.  Two conditions in \cite{LZ2} are quoted here describing the specific orders related to the coalescence rates and the time of coming down from infinity. Moreover, two other conditions describing the fluctuation of the measure $\Lambda$ near $0$ are also stated.
\begin{condition}[Condition A in \cite{LZ2} with $\al=\be-1$]\label{con2}
 There exists a constant $\be>1$ such that the associated $\Lambda$-coalescent $\Pi$ satisfies
$$\limsup_{m\rightarrow\infty}m^{\be-1}\sum_{b=m+1}^{\infty}{\lambda_{b}}^{-1}<\infty.$$
\end{condition}

For any positive integer $m$, let
$\mathbb{T}_m=\inf\{t\geq 0:N(t)\leq m\}$
be the first time that the number of blocks in the associated $\Lambda$-coalescent is less than or equal to $m$
with the convention $\inf\emptyset=\infty$.

\begin{condition}[Assumption I in \cite{LZ2} with $\al=\be-1$]\label{con1}
There exists a constant $\be>1$ such that the associated $\Lambda$-coalescent $\Pi$ satisfies
$$\limsup_{m\rightarrow\infty}m^{\be-1}\mathbb{E} \(\mathbb{T}_m \)<\infty.$$
\end{condition}

\begin{remark}
By \cite[Remark 4.7]{LZ2},  Condition \ref{con2} implies Condition \ref{con1}.
Both are sufficient for the associated $\Lambda$-coalescent to come down from infinity.
\end{remark}

For the main results on the modulus of continuities for the ancestry process, we need to impose  conditions on the behaviors of measure $\Lambda$ near $0$, which determine the small time asymptotic number of blocks in the $\Lambda$ coalescent.
\begin{condition}\label{eq:lalow}
There exist  $\be\in (1,2)$, $A>0$ and $\ep \in (0,1)$ such that 
\beqnn
\La(dx)\geq A x^{1-\be} dx \text{\,\,\,\,for all\,\,\,\,} x\in[0,\ep].
\eeqnn
\end{condition}

\begin{condition}\label{eq:laup}
There exist $\be\in (1,2)$, $A>0$ and $\ep\in(0,1)$ such that 
\beqnn
\La(dx)\leq Ax^{1-\be} dx \text{\,\,\,\,for all\,\,\,\,} x\in[0,\ep].
\eeqnn
\end{condition}

Observe that the Beta$(2-\be,\be)$ coalescent with $\be\in(1,2)$ satisfies both Condition \ref{eq:lalow} and Condition \ref{eq:laup} for the same $\beta$. Condition \ref{eq:lalow} is coincident with the $(c,\ep,\ga)$-property, proposed in \cite{LZ1} for $c=A$ and $\ga=\be-1$.
It indicates that the measure $\La$ has enough mass distributed in a right neighborhood of $0$ to cause fast merging so that the associated coalescent comes down from infinity. The upper bound in Condition \ref{eq:laup} is comparable to the Beta function from (\ref{eq:beta1}). It indicates that the collision speed of the associated coalescent can not exceed that of Beta-coalescent in a small positive time.

\begin{remark}\label{con:compare}
If $\La$ has a non-trivial component at $0$, i.e. $\La(\{0\})>0$,  then Condition \ref{con2} holds with $\be=2$;
if $\La(\{0\})=0$ and Condition \ref{eq:lalow} holds, then by \cite[Lemma 4.13]{LZ1}, Condition \ref{con2} holds with $\be\in(1,2)$.
\end{remark}

\section{ $\Lambda$-Fleming-Viot process and its lookdown representation}\label{sec:lookdown}
In this section, we first discuss the lookdown representation of $\La$-Fleming-Viot process with Brownian spatial motion. Then we explain
how to recover the $\La$-coalescent and the ancestry process from the lookdown representation.

\subsection{Lookdown representation}
Donnelly and Kurtz \cite{DK96,DK99a,DK99b} introduced the lookdown representation, a discrete representation for measure-valued stochastic process.
It provides a unified framework by simultaneously describing the forwards evolution and the ancestry process of the approximating particle system.
We  follow Birkner and Blath \cite{BB} to give a brief review on the lookdown representation of the $\Lambda$-Fleming-Viot process with Brownian spatial motion. Considering
a countably infinite number of particles, each of them is attached a ``level" from the set $[\infty]=\{1,2,\ldots\}$. The evolution of a particle at level $n$ only depends on the evolution of the finite particles at lower levels, which allows the system to have a projective property. Let $\(X_1(t),X_2(t),X_3(t),\ldots\)$ be an $(\mathbb{R}^d)^{\infty}$-valued random variable, where $X_i(t)$ represents the spatial location of the particle at level $i$ with $i\in[\infty]$. The initial values $\left\{X_i(0),i\in[\infty]\right\}$ are exchangeable so that the limiting empirical measure $\lim_{n\rightarrow\infty}\frac{1}{n}\sum_{i=1}^n\delta_{X_i(0)}$ exists almost surely by de Finetti's theorem.

Denote by $\La_0$ the measure $\La$ restricted to $(0,1]$ to determine the multiple birth events. The remaining part $\La(\{0\})\delta_0$ determines the single birth events.

For the single birth events, let $\{\mathbf{N}_{ij}(t): 1\leq i<j<\infty\}$ be independent Poisson processes with rate $\La(\{0\})$. At a jump time $t$ of $\mathbf{N}_{ij}$, the particle at level $j$ looks down at the particle at level $i$ and copies the type there (the new particle at level $j$ can be viewed as the child of the original particle at level $i$). Types on  levels above $j$ are shifted accordingly, i.e., if $\Delta{\mathbf{N}}_{ij}(t)=1$, we have
\begin{eqnarray*}
X_k(t)\,=\,\begin{cases} X_k(t-), &\text{~~if $k<j$},\cr
X_i(t-), &\text{~~if $k=j$},\cr
X_{k-1}(t-), &\text{~~if $k>j$}.
\end{cases}
\end{eqnarray*}

The multiple birth events can be constructed by an independent Poisson point process $\tilde{\mathbf{N}}$ on $\mathbb{R}_{+}\times\left(0, 1\right]$ with intensity measure $dt\otimes x^{-2}\La_0\(dx\)$. Let $\{U_{ij}$, $i, j\in[\infty]\}$ be i.i.d. uniform $[0, 1]$ random variables. Jump points $\{\(t_i, x_i\)\}$ for $\tilde{\mathbf{N}}$ correspond to the multiple birth events. For $t\geq 0$ and  $J\subseteq [n]$ with $|J|\geq 2$, define
\begin{eqnarray*}
\mathbf{N}_J^n(t)\,:=\,\sum_{i:t_i\leq t}\prod_{j\in
J}\mathbf{1}_{\{U_{ij}\leq x_i\}}\prod_{j\in[n]\backslash
J}\mathbf{1}_{\{U_{ij}> x_i\}},
\end{eqnarray*}
where $\mathbf{1}_{\{\cdot\}}$ denotes an indicator function. Then $\mathbf{N}_J^n(t)$ counts the number of birth events among the particles from
levels $\{1,2,\ldots,n\}$ such that exactly those at levels in $J$ are involved up to time $t$. Each participating level adopts the type of the lowest level involved. All the other individuals are shifted upwards accordingly, keeping their original order. More specifically, if $t=t_i$ is a jump time and $j$ is the lowest level involved, then
\begin{eqnarray*}
X_k(t)\,=\,
\begin{cases}
X_k(t-), \text{~for~} k\leq j,\\
X_j(t-), \text{~for~} k>j \text{~with~} U_{ik}\leq x_i,\\
X_{k-J_t^k}(t-), \text{~otherwise},
\end{cases}
\end{eqnarray*}
where $J_{t}^k:=\#\{m<k, U_{im}\leq x_i\}-1$.

Between jump times of the Poisson point process, particles at different levels move independently according to Brownian motions in $\mathbb{R}^d$.

It is shown in \cite{DK99b} that for each $t>0$, $X_1(t),X_2(t),\ldots$ are exchangeable;
$$X(t):=\lim_{n\rightarrow\infty}X^{(n)}(t):= \lim_{n\rightarrow\infty}\frac{1}{n}\sum_{i=1}^n\delta_{X_i(t)}$$
exists almost surely by  de Finetti's theorem and $X\equiv\(X(t)\)_{t\geq 0}$ is right continuous and follows the
probability law of the $\La$-Fleming-Viot process with Brownian spatial motion. Write $M_1(\mathbb{R}^d)$ for the space of probability measures on $\mathbb{R}^d$ equipped with the topology of weak convergence.  One can show that $(X^{(n)}(t))_{t\geq 0}$ converges almost surely in $D_{M_1(\mathbb{R}^d)}([0, \infty)) $ to $(X(t))_{t\geq 0}$ that is a strong Markov process; see \cite[Theorem 1.1]{BB}.

\subsection{$\Lambda$-coalescent in the lookdown representation}
For any $0\leq s\leq t$, denote by $L_n^t(s)$ the ancestor's level at time
$s$  for the particle with level $n$ at time $t$. Write $\big(\Pi^t(s)\big)_{0\leq s\leq t}$ for the
$\mathcal{P}_\infty$-valued process such that $i$ and $j$ belong to
the same block of $\Pi^t(s)$ if and only if $L_i^t(t-s)=L_j^t(t-s)$.
The process $\big(\Pi^t(s)\big)_{
0\leq s\leq t}$ has the same law as the
$\La$-coalescent running up to time $t$.
For any $0\leq s\leq t$, write
\beqlb\label{def:N}
N({s,t})\ar:=\ar\#\Pi^{t}\(t-s\)
\eeqlb
for the number of blocks in $\Pi^{t}\(t-s\)$ and
$\Pi^{t}\(t-s\):=\{\pi^{s,t}_{\ell}: 1\leq \ell\leq N({s,t})\},$
where $\pi^{s,t}_{\ell}, 1\leq \ell\leq N({s,t})$ are all the
disjoint blocks ordered by their least
elements. Let
\beqlb\label{eq:comdown}
\mathbb{T}^t_m:=\inf\{ 0\leq s\leq t:\,\#\Pi^t(s)\leq m\}
\eeqlb
with the convention $\inf\emptyset=t$.

\subsection{Ancestry process}
For any $T>0$,
denote by
$\(X_{1,t},X_{2,t},X_{3,t},\ldots\)_{0\leq t\leq T}$
the ancestry process recovered from the lookdown representation
with $X_{i,t}$ defined by
\beqnn
X_{i,t}\(s\)\ar:=\ar
X_{L_i^t\(s\)}\(s-\) \text{~~for~~}0\leq s\leq t.
\eeqnn
By the definition, $\(X_{i,t}(s)\)_{0\leq s\leq t}$ keeps a record of the trajectory for the $i$-th particle alive at time $t$
and $\(X_{1,t},X_{2,t},X_{3,t},\ldots\)$ depicts the genealogy of all those particles alive at time $t$.
We call $\(X_{1,t},X_{2,t},X_{3,t},\ldots\)$ the ancestry process recovered at time $t$.
For any $0\leq r< s\leq t$, let $H^{t}(r,s)$ be the maximal dislocation  between times $r$ and $s$  of the ancestry process
recovered at time $t$ defined by
\beqlb\label{eq:max}
H^{t}(r,s)\ar:=\ar\max_{j\in [\infty]}\left|X_{L_j^{t}(s)}(s-)-X_{L_j^{t}(r)}\(r-\)\right|\cr
\ar=\ar\max_{1\leq \ell\leq
N(r,t)}\max_{j\in\pi^{r,t}_{\ell}}\left|X_{L_j^{t}(s)}\(s-\)-X_{\ell}(r-)\right|,
\eeqlb
where the last equation follows from the fact that the ancestors always come from those particles with lower levels (\cite[Lemma 3.2]{LZ1}).
Therefore, $H^{t}(r,s)$ describes the maximal displacement between two sequential ancestors at times $r$ and $s$ (with one being the ancestor of the other) of those particles alive at $t$.
If $s=t$ in (\ref{eq:max}), then $H^{t}(r,t)$  can be interpreted as the maximal dislocation of the ancestry process during time period $[r,t]$.

\section{Main results}\label{sec:theorems}
Let $\(\Omega,\mathcal{F},\mathcal{F}_t,\mathbb{P}\)$
be a filtered probability space, on which the $\Lambda$-Fleming-Viot process $\(X(t)\)_{t\geq0}$ with Brownian spatial motion is defined.
For any $\nu\in M_1(\mathbb{R}^d)$ write $\mathbb{P}_{\nu}$ for the law of  $\(X(t)\)_{t\geq0}$ with initial state $\nu$.
Recall the modulus function $h(x)=\sqrt{x\log\(1/x\)}$ for any $x>0$. We remark that  in all of the following results,  no additional condition is needed for the case $\Lambda(\{0\})>0$ and the related modulus of continuities hold for $\beta=2$. The proofs of these results are deferred to Section \ref{sec:result}.

\begin{theorem}\label{co:moduglobal}
Assume that either $\La(\{0\})>0$ holds or $\La(\{0\})=0$ together with  Conditions \ref{eq:lalow} and \ref{eq:laup} hold for the same $\be$ and possibly different constants $A$. Then for any $T>0$ and any $\nu\in M_1(\mathbb{R}^d)$, we have $\mathbb{P}_\nu$-a.s.	 \begin{equation}\label{uniform_modu}
	\limsup_{\ep\to 0+}\sup_{\ep\leq t\le T }\frac{H^t\({t-\ep,t}\)}{h(\ep)}=\sqrt{\frac{2\beta}{\beta-1}}.
\end{equation}
\end{theorem}

\begin{remark}
As far as we know, the only work related to modulus of continuity for Fleming-Viot process has been proposed in our previous paper \cite{LZ2},  where we have just derived that the left side of (\ref{uniform_modu}) is bounded by a constant $C^{*}$.
See Lemma \ref{le:mod1} and Remark \ref{re:C} below for detailed descriptions. The current result presents the exact value and has greatly improved the previous work.
One can also compare (\ref{uniform_modu}) with the exact modulus of continuity for standard $d$-dimensional Brownian motion $(B(t))_{t\geq 0}$ such that
	$$\limsup_{\ep\to 0+}\sup_{\ep\leq t\le T }\frac{|B(t)-B(t-\ep)|}{h(\ep)}=\sqrt{2} \,\,a.s. $$
See e.g. M\"orters and Peres \cite[Theorem 1.14]{MP}. A similar sharp global modulus of continuity  can be found in Dawson et al. \cite[Theorem 4.7]{DIP} for historical sample paths of super-Brownian motion that corresponds to the  Kingman-Fleming-Viot process.
\end{remark}

\begin{theorem}\label{co:modu}
Assume that either $\La(\{0\})>0$ holds or $\La(\{0\})=0$ together with Conditions \ref{eq:lalow} and \ref{eq:laup} hold for the same $\be$ and possibly different constants $A$.
Then for fixed $t>0$, $s\geq 0$ and any $\nu\in M_1(\mathbb{R}^d)$, we have $\mathbb{P}_\nu$-a.s.
	\begin{equation*}
	\limsup_{\ep\to 0+}\frac{H^t\({t-\ep,t}\)}{h(\ep)}=\sqrt{\frac{2}{\beta-1}}
\end{equation*}	
and
	\begin{equation*}
	\limsup_{\ep\to 0+}\frac{H^{s+\ep}\({s, s+\ep}\)}{h(\ep)}=\sqrt{\frac{2}{\beta-1}}.
\end{equation*}	
\end{theorem}

The above results can be applied to obtain new results concerning  modulus of continuity for the support process.
Given $\eta>0$, for any Borel set ${A}\subset\mathbb{R}^d$, let
$\mathbb{B}\({A},\eta\)\equiv\overline{\bigcup_{x\in A}\mathbb{B}\(x,\eta\)}$ be its closed $\eta$-neighborhood with $\mathbb{B}\(x,\eta\)$ being the closed ball centered at $x$ with radius $\eta$.
Denote by $S(\mu)$ the closed support of a measure $\mu$. For any compact subsets $K_1, K_2$ of $\mathbb{R}^d$ let
$$\rho_1(K_1, K_2):=\sup_{x\in K_1}d(x, K_2)\wedge 1\,\,\text{and}\,\,\rho (K_1, K_2):=\rho_1(K_1, K_2)\vee  \rho_1(K_2, K_1)$$
with $\rho(K_1, \phi):=1.$
Then $\rho(\cdot, \cdot)$ is the Hausdorff metric on the space of compact sets in $\mathbb{R}^d$.

\begin{corollary}\label{co:suppor}
Assume that Condition \ref{con2} holds.
\begin{enumerate}
\item\label{co:global}
Given any $T>0$, $\nu\in M_1(\mathbb{R}^d)$ and  $c>\sqrt{2\be/(\be-1)}$, we have $\mathbb{P}_{\nu}$-a.s.
	$S(X(t))\subseteq\mathbb{B}(S(X(t-\ep)), ch(\ep))$ for all $0\leq t-\ep<t\leq T$ and small enough $\ep>0$.
\item\label{co:localleft}
For each fixed $t>0$, any $\nu\in M_1(\mathbb{R}^d)$ and $c>\sqrt{2/(\be-1)}$, we have $\mathbb{P}_{\nu}$-a.s. $S(X(t))\subseteq\mathbb{B}(S(X(t-\ep)), ch(\ep))$ for all small enough $\ep>0$.
\item\label{co:localright}
For each fixed $s\geq0$, any $\nu\in M_1(\mathbb{R}^d)$ and $c>\sqrt{2/(\be-1)}$, we have $\mathbb{P}_{\nu}$-a.s.
		$S(X(s+\ep))\subseteq\mathbb{B}(S(X(s)), ch(\ep))$ for all small enough $\ep>0$.
\item\label{co:righcon}
For any $\nu\in M_1(\mathbb{R}^d)$, we have $\mathbb{P}_\nu$-a.s. $S(X(t))$ is right continuous in the Hausdorff metric for all $t>0$.
\end{enumerate}
\end{corollary}

For any $t>0$ write
$$\rho(t):=\inf \{r\geq 0: S(X(t))\subseteq \mathbb{B}(0, r) \}$$
for the maximal dislocation (from the origin) that the support can reach at time $t$.
The next result gives the rate of propagation of support when started with a point mass situated at  $0$. A similar result
for super-Brownian motion can be found in Tribe \cite[Theorem 2.1]{Tri89}.

\begin{corollary}\label{co:rho}
	Assume that either $\La(\{0\})>0$ holds or $\La(\{0\})=0$ together with Conditions {\ref{eq:lalow}} and {\ref{eq:laup}} hold for the same $\be$ and possibly different constants $A$. Given $X(0)=\delta_0$, the Dirac mass at $0$,   we have $\mathbb{P}_{\delta_{0}}$-a.s.
	$$\lim_{t\to 0+}\frac{\rho(t)}{h(t)}=\sqrt{\frac{2}{\be-1}}. $$
\end{corollary}

\begin{remark}
It is easy to show that the Beta$(2-\be,\be)$ coalescent with $\be\in(1,2]$  satisfies Condition \ref{con2}. Moreover, the Beta$(2-\be,\be)$ coalescent with $\be\in(1,2)$ satisfies  Conditions \ref{eq:lalow} and \ref{eq:laup} for the same $\beta$. Therefore, all the results in Theorems \ref{co:moduglobal}, \ref{co:modu} and Corollaries \ref{co:suppor}-\ref{co:rho} hold for
the Beta$(2-\beta,\beta)$-Fleming-Viot process with $\be\in(1,2]$.
\end{remark}

\section{Estimates, proofs and applications}\label{sec:result}
In this section, we first carry out some estimations on  the $\La$-coalescent, and then recall some known results in the literature. After that we prove Theorems \ref{co:moduglobal} and \ref{co:modu}
concerning both the global and local sharp modulus of continuities for the ancestry process recovered from the lookdown representation. Finally, the  modulus of continuity results (Corollaries \ref{co:suppor}-\ref{co:rho}) related to the $\Lambda$-Fleming-Viot support process are verified.

\subsection{Some estimates on the number of blocks of the $\La$-coalescent}
We begin with recalling an equivalent condition for the $\La$-coalescent to come down from infinity (see Bertoin and Le Gall \cite{BerLeGa06}). Let
\beqlb\label{eq:comingdown}
\psi(q)\ar:=\ar\int_{[0,1]}(e^{-qx}-1+qx)x^{-2}\Lambda(dx)\cr
\ar=\ar\La(\{0\})\frac{q^2}{2}+\int_{(0,1]}(e^{-qx}-1+qx)x^{-2}\Lambda(dx),
\eeqlb
then
\beqnn
\sum_{n=2}^{\infty}\gamma_n^{-1}<\infty \text{\,\, if and only if \,\,} \int_a^\infty \frac{dq}{\psi(q)}<\infty,
\eeqnn
where the integral is finite for some (and then for all) $a>0$.
If the associated $\La$-coalescent comes down from infinity, define
\beqnn
u(t):=\int_t^{\infty}\frac{dq}{\psi(q)},\,\,\,\,t>0
\eeqnn
and its c\`adl\`ag inverse
\beqlb\label{eq:v}
v(t):=\inf\left\{s>0:\int_s^{\infty}\frac{dq}{\psi(q)}<t\right\},\,\,\,\,t>0.
\eeqlb
Berestycki et al. \cite{BBL} showed that $\lim_{t\rightarrow 0}N(t)/v(t)=1$ almost surely conditional on $\La([0,1))=1$. A similar result  also holds for any $\La$ satisfying $\La([0,1))<\infty$ by adjusting the term by a scale coefficient. Some specific estimations on $v(t)$ are derived in the following proposition for later use. 

\begin{proposition}\label{pro:v}
Given any finite measure $\La$ on the unit interval, we have
\beqlb\label{eq:v2}
v(t)\geq{2}\(\La\([0,1]\)\)^{-1}t^{-1}\text{\,\,\,\,for all\,\,\,\,}t> 0.
\eeqlb
Moreover, if
$\La(\{0\})=0$ and Condition \ref{eq:laup} 
holds, then there exists a constant $C(A,\ep,\be)$ such that
\beqlb\label{eq:v1}
v(t)\geq  C(A,\ep,\be) t^{-{1}/{(\be-1)}} \text{\,\,\,\,for all\,\,\,small $t> 0$}.
\eeqlb
\end{proposition}
\begin{proof}
Clearly, $v(t)$ is a  non-increasing function.
Since
$e^{-qx}\leq 1-qx+q^2x^2/2$
for $x>0$, then
\beqnn
\psi(q)\ar \leq\ar\La(\{0\})\frac{q^2}{2}+\frac{\La\((0,1]\)q^2}{2}=\frac{\La\([0,1]\)q^2}{2}.
\eeqnn
Hence,
\beqnn
u(t)\geq\int_t^{\infty}\frac{2dq}{\La\([0,1]\)q^2}=\frac{2}{\La\([0,1]\)t}.
\eeqnn
Thus, Equation (\ref{eq:v2}) follows from (\ref{eq:v}).

If $\La(\{0\})=0$ and Condition \ref{eq:laup} holds for $1<\beta<2$, then there exist constants $A>0$  and $\ep\in(0,1]$ such that for all $x\in(0,\ep]$,
$\La(dx)\leq Ax^{1-\be}dx,$
which combined with (\ref{eq:comingdown}) gives that
{\small\beqlb\label{eq:psi1}
\psi(q)\ar=\ar\int_{(0,\ep]}\(e^{-qx}-1+qx\)x^{-2}\La(dx)+\int_{(\ep,1]}\(e^{-qx}-1+qx\)x^{-2}\La(dx)\cr
\ar\leq\ar A \int_{(0,\ep]}\(e^{-qx}-1+qx\)x^{-1-\be}dx+\int_{(\ep,1]}\(e^{-qx}-1+qx\)x^{-2}\La(dx)\cr
\ar=\ar A q^{\be}\int_{(0,q\ep]}\(e^{-y}-1+y\)y^{-1-\be}dy+\int_{(\ep,1]}\(e^{-qx}-1+qx\)x^{-2}\La(dx).
\eeqlb
}
Moreover, as $y$ tends to $0$, $e^{-y}-1+y\sim y^2/2$, where $\sim$ means that the ratio of the two sides tends to $1$ as $y\downarrow 0$. As a result,
$\(e^{-y}-1+y\)y^{-1-\be}\sim y^{1-\be}/2,$
which is integrable in a right neighborhood of $0$ due to $1<\be <2$. As $y$ tends to $\infty$, we see that functions
$e^{-y}y^{-1-\be}$, $y^{-1-\be}$ and $y^{-\be}$ are all integrable in a neighborhood of $\infty$. Therefore,
\beqnn
\left|\int_{(0,q\ep]}\(e^{-y}-1+y\)y^{-1-\be}dy\right|\ar\leq\ar \int_{(0,\infty)}\(e^{-y}-1+y\)y^{-1-\be}dy<\infty.
\eeqnn
Observe  that $\(e^{-qx}-1+qx\)x^{-2}$ is continuous on $(\ep,1]$ and $\La((\ep,1])<\infty$, which implies
\beqnn
\left|\int_{(\ep,1]}\(e^{-qx}-1+qx\)x^{-2}\La(dx)\right|<\infty.
\eeqnn
Thus, there exists some constant $\widetilde{C}(A,\ep,\be)$ so that for $q$ large enough, $\psi(q)\leq \widetilde{C}( A,\ep,\be)q^{\be}$. This fact, combined with (\ref{eq:v}), gives
\beqnn
v(t)\geq \(\widetilde{C}(A,\ep,\be)\)^{-1/(\be-1)}t^{-1/(\be-1)}
\eeqnn
as $t$ tends to $0$.
Setting ${C}(A,\ep,\be)=\(\widetilde{C}(A,\ep,\be)\)^{-1/(\be-1)}$ we arrive at the desired result. The proof is complete.
\end{proof}

In the following we state a short time estimate on the lower bound for the number of blocks in a $\Lambda$-coalescent. It is an adaption of Berestycki et al. \cite[Proposition 12]{BBL}. 

\begin{proposition}\label{pro:ancestor}
Let $\La$ be a finite measure on the unit interval, $\Pi$ be the associated $\Lambda$-coalescent and $N(s)$ be the number of blocks in $\Pi(s)$ for $s>0$. Then for any $\al^{*}\in(0,1/2)$, we have
\beqlb\label{eq:N2}
\mathbb{P}\(N\(s\)\geq e^{-24  s^{\al^{*}}}{v(s)}\)\geq 1- O\(s^{1-2\al^{*}}\)\quad\text{as}\quad s\to 0,
\eeqlb
 where $O\(s^{1-2\al^{*}}\)$ is the term with the same order as $s^{1-2\al^{*}}$ and $v(s)$ is defined by (\ref{eq:v}).
\end{proposition}
\begin{proof}
We start the proof with the assumption that supp$\(\La\)\subset [0,1/4]$. It follows from \cite[Proposition 12]{BBL} that
\beqnn
\mathbb{P}\(\sup_{t\in(0,s\wedge\mathbb{T}_{n_0}]}\left|\log\frac{N(t)}{v(t)}\right|\leq 24 \times s^{\al^{*}}\)
\ar\geq\ar 1-O\(s^{1-2\al^{*}}\),
\eeqnn
where $\al^{*}\in(0,1/2)$, $n_0$ is a positive integer given in \cite[Lemma 19]{BBL} and $\mathbb{T}_{n_0}$ is the first time that the number of blocks is less than or equal to $n_0$. Then we have
\beqnn
\mathbb{P}\(\log\frac{N\(s\wedge\mathbb{T}_{n_0}\)}{v(s\wedge\mathbb{T}_{n_0})}\geq -24 \times s^{\al^{*}}\)
\ar\geq\ar 1-O\(s^{1-2\al^{*}}\).
\eeqnn
Equivalently,
\beqlb\label{eq:N}
\mathbb{P}\({N\(s\wedge\mathbb{T}_{n_0}\)}\geq e^{-24 \times s^{\al^{*}}}{v(s\wedge\mathbb{T}_{n_0})}\)
\ar\geq\ar 1-O\(s^{1-2\al^{*}}\).
\eeqlb
Observe that $N\(s\wedge\mathbb{T}_{n_0}\)>n_0$ if and only if $N\(s\)>n_0$.
By (\ref{eq:v2}), $v(s)\rightarrow+\infty$ as $s\downarrow 0$. Combining this result with the fact that $\lim_{s\rightarrow 0}e^{-24 \times s^{\al^{*}}}=1$, we see that ${N\(s\wedge\mathbb{T}_{n_0}\)}>n_0$ for all small enough $s>0$.
Then as $s\rightarrow 0$, (\ref{eq:N2}) is derived in case of  supp$\(\La\)\subset [0,1/4]$.

For general finite measure $\La$ on the unit interval, recall that  $\La=\La(\{0\})\delta_0+\La_0$. The collisions for the $\Lambda$-coalescent are due to a binary merger determined by $\La(\{0\})\delta_0$ and a multiple merger determined by $\La_0$.
The multiple merger can be constructed by Poisson point processes. Let $\Psi_1$ and $\Psi_2$ be independent Poisson point processes on $(0,\infty)\times \{0, 1\}^\infty$ with intensity $dtL_1(d\xi)$ and $dtL_2(d\xi)$, respectively, where
	$$L_1(d\xi)=\int_{(0,1/4]} x^{-2}\La(dx)Q_{x}(d\xi)\quad\text{and}\quad L_2(d\xi)=\int_{(1/4,1]} x^{-2}\La(dx)Q_{x}(d\xi).$$
The total mass of $L_2$ is $\int_{(1/4,1]}x^{-2}\La(dx)<\infty$. Let $Z(s)$ be the number of points in $\Psi_2$ until time $s$. Then
\beqlb\label{eq:N5}
\mathbb{P}\(Z(s)=0\)=e^{-s\int_{(1/4,1]}x^{-2}\La(dx)}.
\eeqlb
As $s$ tends to $0$,
\beqnn
\ar\ar\mathbb{P}\({N\(s\)}\leq e^{-24 \times s^{\al^{*}}}{v(s)}\)\cr
\ar\ar\quad=\mathbb{P}\({N\(s\)}\leq e^{-24 \times s^{\al^{*}}}{v(s)},Z(s)\geq1\)
+\mathbb{P}\({N\(s\)}\leq e^{-24 \times s^{\al^{*}}}{v(s)},Z(s)=0\)\cr
\ar\ar\quad\leq \mathbb{P}\(Z(s)\geq1\)+\mathbb{P}\({N\(s\)}\leq e^{-24 \times s^{\al^{*}}}{v(s)},Z(s)=0\).
\eeqnn
The estimation for the second term in the above line can be attributed to the case of  supp$\(\La\)\subset [0,1/4]$.
This fact, combined with (\ref{eq:N5}), yields
\beqnn
\mathbb{P}\({N\(s\)}\leq e^{-24 \times s^{\al^{*}}}{v(s)}\)\ar\leq\ar 1-e^{-s\int_{(1/4,1]}x^{-2}\La(dx)}+O\(s^{1-2\al^{*}}\)\cr
\ar\leq \ar{s\int_{(1/4,1]}x^{-2}\La(dx)}+O\(s^{1-2\al^{*}}\)\cr
\ar=\ar O\(s^{1-2\al^{*}}\).
\eeqnn
The proof is complete.
\end{proof}

\subsection{Previous results in the literature}
We begin with recalling a relatively loose upper modulus of continuity result in \cite{LZ2} for the ancestry processes recovered from the lookdown representation of a class of $\Lambda$-Fleming-Viot processes with Brownian spatial motion.
\begin{lemma}[cf. Theorem 4.1 and Remark 5.4 in \cite{LZ2}]\label{le:mod1}
Under Condition \ref{con1} and for any $T>0$, there exist a positive random
variable $\theta^{*}\equiv\theta^{*}\(T,d,\be\)<1$ and a constant $C^{*}\equiv
C^{*}(d,\be)$ such that $\mathbb{P}$-a.s. for all $r,s\in [0,T]$
satisfying $0<s-r\leq\theta^{*} $, we have
\begin{eqnarray}\label{eq:mod1}
\begin{split}
H^s\(r,s\) \,\leq\, C^{*}\sqrt{\(s-r\)\log\(1/\(s-r\)\)}.
\end{split}
\end{eqnarray}
Further, there exist positive
constants $C_6= C_6(T, d,\be)$ and $C_7= C_7(d,\be)$ such
that for $\epsilon>0$ small enough
\beqlb\label{eq:theta}
\mathbb{P} (\theta^{*} \leq\epsilon)\,\leq\, C_6 {\epsilon}^{C_7}.
\eeqlb

\end{lemma}

\begin{remark}\label{re:C}
It follows from  \cite[Remark 5.6]{LZ2} that
\beqnn
C^{*}\ar\geq\ar2\sqrt{2d\(3/(\be-1)+1\)}\Big(1+\sum_{\ell=1}^{\infty}\sqrt{2^{-2\ell+1}\ell}\Big)
\Big(1+\sum_{\ell=1}^{\infty}\sqrt{2^{-\ell+1}\ell}\Big)\cr
\ar\approx\ar19.53457\sqrt{2d\(3/(\be-1)+1\)}.
\eeqnn
If we make the estimation in the proofs of \cite[Lemma 5.2 and Theorem 4.1]{LZ2} more precise,
the lower bound for $C^*$ is strictly greater than
$(12+4\sqrt{2})\sqrt{2d(3/(\be-1)+1)}$.
In particular, $C^{*}>(24\sqrt{2}+16)\sqrt{d}$ if $\be=2$. Compared with the exact modulus of continuity of super-Brownian motion, whose support properties are similar to the Kingman-Fleming-Viot process,   the value of $C^{*}$ is far from optimal.
\end{remark}

\begin{remark}
Under Condition \ref{con1}, a one-sided modulus of continuity is proved for the $\Lambda$-Fleming-Viot support process at fixed time, and the support is shown to be almost surely compact  at all positive time; see  \cite[Theorems 4.2 and 4.3]{LZ2}.
\end{remark}

\begin{lemma}\label{le:BM}
For standard $d$-dimensional Brownian motion $\(B(t)\)_{t\geq 0}$, we have
\beqlb\label{eq:bm}
\mathbbm{c}_{1}(d)R^{d-2}e^{-R^2/2}\leq \mathbb{P}\(|B(1)|\geq R\)\leq \mathbbm{c}_2(d)R^{d-2}e^{-R^2/2}
\eeqlb
holds for  $R$ large enough, where $\mathbbm{c}_1(d)$ and $\mathbbm{c}_2(d)$ are constants depending on $d$.
\end{lemma}
This is an estimation for the tail  probability of standard $d$-dimensional Brownian motion.
It can be found in \cite[Equation (3.9)]{Da94} or \cite[Lemma 3]{DL98}.

\subsection{Modulus of continuities for the ancestry process}
We proceed to derive both the  global and local modulus of continuities for the ancestry process recovered from the lookdown representation of the $\Lambda$-Fleming-Viot process.
\subsubsection{Estimates on the maximal dislocation between two  fixed times that are close to each other}
Given any fixed time $t$ and a time a bit ahead of it, i.e. $t-\theta^nu$, the time interval is divided
into countably infinite many subintervals $\{[t-\theta^{k-1}u,t-\theta^{k}u]:k> n\}$ with their lengths obeying a geometric sequence.
Looking backwards in time, we can recover the ancestry process
at time $t$ from the lookdown representation. We first estimate the probability that the dislocation of the recovered ancestry process during a subinterval $[t-\theta^{k-1}u,t-\theta^{k}u]$ is larger than a bound (comparable to the interval length $\theta^{k-1}u(1-\theta)$), which is shown to be summable with respect to $k$. Then assuming that the maximal dislocation of the ancestry process during each subinterval $[t-\theta^{k-1}u,t-\theta^{k}u]$ with $k>n$ is always smaller than the designated bound, an upper bound for the maximal dislocation between $t$ and $t-\theta^n u$ is naturally derived by the sum of bounds on all subintervals where detailed arguments are carried out.

For any $\theta\in(0,1)$, $u\in(0,1]$, $c>0$ and $t>0$, let
\beqnn
A_k\ar \equiv\ar A_k(t,u,\theta,c)=\left\{H^{t}\({t-\theta^{k-1}u,t-\theta^{k}u}\)>ch\(\theta^{k-1}u(1-\theta)\)\right\}
\eeqnn
be the random event that dislocation of the recovered ancestry process during subinterval $[t-\theta^{k-1}u,t-\theta^{k}u]$ is larger than the designated bound.
Denote by
\beqlb\label{eq:B_n}
B_n\ar \equiv \ar B_n(t,u,\theta,c)=\bigcup_{k>n,\,t>\theta^{k-1}u}A_k(t,u,\theta,c).
\eeqlb
Note that $t>\theta^{k-1}u$ in the above union stresses that the value of $k$ is large enough to make $A_k(t,u,\theta,c)$ meaningful.
Observe that $B_n$ is decreasing  in $n$. Actually, $B_n$ is the random event that among all the ancestors at the grid times of those particles alive at time $t$, there exists at least two  consecutive ancestors  with their dislocation larger than the designated bound during a subinterval.

\begin{lemma}\label{le:mod}
Under Condition \ref{con2}, let $\theta$, $u$ and $t$ be as above and let $c>\sqrt{2/{(\be-1)}}$.
\begin{enumerate}
\item\label{le:moda} For any $\nu\in M_1(\mathbb{R}^d)$, there are two constants $C(c, d, \be, u,\theta)$ and $\bar{C}(C,u,\theta)$ such that
\beqnn
\mathbb{P}_{\nu}\(B_n\)\leq C(c,d,\be, u,\theta)\theta^{\(\frac{c^2}{2}-\frac{1}{{\be-1}}\)n}n^{\frac{2}{{\be-1}}+\frac{d}{2}-1}+\bar{C}(C,u,\theta)e^{-\frac{n}{\theta}}.
\eeqnn
\item\label{le:modb} For any  $\epsilon>0$, there exists an $n_0=n_0\(\epsilon, u, \theta\)$ such that for any $n\geq n_0$  satisfying $t\geq \theta^n u$, if $B_n$ does not occur, then we have
\beqnn
H^t({t-\theta^nu,t})\leq c(1+\epsilon)(1-\theta)^{\frac12}(1-\theta^{\frac12})^{-1}h(\theta^nu).
\eeqnn
\end{enumerate}
\end{lemma}

\begin{proof}
\ref{le:moda} By the total probability formula, the probability of $A_k$ satisfies
\beqnn
\mathbb{P}_{\nu}(A_k)\ar=\ar
\mathbb{P}_{\nu}\(A_k,N({t-\theta^ku,t})\leq\theta^{-\frac{k}{{\be-1}}}k^{\frac{2}{{\be-1}}}\)+\mathbb{P}_{\nu}\(A_k,N({t-\theta^ku,t})>\theta^{-\frac{k}{{\be-1}}}k^{\frac{2}{{\be-1}}}\)\cr
\ar\leq\ar\left\lfloor\theta^{-\frac{k}{{\be-1}}}k^{\frac{2}{{\be-1}}}\right\rfloor
\times\mathbb{P}\(|B(\theta^{k-1}u\(1-\theta\))|>ch\(\theta^{k-1}u\(1-\theta\)\)\)\cr
\ar\ar\qquad+\mathbb{P}_{\nu}\(N({t-\theta^ku,t})>\theta^{-\frac{k}{{\be-1}}}k^{\frac{2}{{\be-1}}}\)\cr
\ar\equiv\ar I_{1,k}+I_{2,k},
\eeqnn
where $B(\cdot)$ is the standard Brownian motion.
By estimate (\ref{eq:bm}),
\beqnn
I_{1,k}\ar\leq\ar\left\lfloor\theta^{-\frac{k}{{\be-1}}}k^{\frac{2}{{\be-1}}}\right\rfloor
\times\mathbb{P}\Big(|B(1)|>c\sqrt{\log\(\theta^{-k+1}u^{-1}\(1-\theta\)^{-1}\)}\Big)\cr
\ar\leq\ar\theta^{-\frac{k}{{\be-1}}}k^{\frac{2}{{\be-1}}}
\times \mathbbm{c}_2(d)c^{d-2}\left[\log\(\theta^{-k+1}u^{-1}(1-\theta)^{-1}\)\right]^{\frac{d}{2}-1}\(\theta^{k-1}u(1-\theta)\)^{\frac{c^2}{2}}\cr
\ar\leq \ar C_1(c,d,u,\theta)\theta^{\(\frac{c^2}{2}-\frac{1}{{\be-1}}\)k}k^{\frac{2}{{\be-1}}}\left[\log\(\theta^{-k+1}u^{-1}(1-\theta)^{-1}\)\right]^{\frac{d}{2}-1}.
\eeqnn
Moreover, it follows from \cite[Lemma 5.8]{LZ2} that
\beqnn
I_{2,k}\ar=\ar\mathbb{P}_{\nu}\(\mathbb{T}^{t}_{\big\lfloor\theta^{-\frac{k}{{\be-1}}}k^{\frac{2}{{\be-1}}}\big\rfloor}>\theta^{k}u\)
\,\leq\,\mathbb{P}_{\nu}\Bigg(\sum_{i>\big\lfloor\theta^{-\frac{k}{{\be-1}}}k^{\frac{2}{{\be-1}}}\big\rfloor}\hat{\tau_i}>\theta^{k}u\Bigg),
\eeqnn
where $\mathbb{T}^{t}_m$ is defined by (\ref{eq:comdown}) and $\hat{\tau_i}$, $i>\big\lfloor\theta^{-\frac{k}{{\be-1}}}k^{\frac{2}{{\be-1}}}\big\rfloor$ are independent exponentially distributed random variables with parameter $\lambda_i$. Then
\beqlb\label{eq:I_2}
I_{2,k}\ar\leq\ar\mathbb{P}_{\nu}\Bigg(\theta^{-k}u^{-1}\sum_{i>\big\lfloor\theta^{-\frac{k}{{\be-1}}}k^{\frac{2}{{\be-1}}}\big\rfloor}\hat{\tau_i}>1\Bigg)\cr
\ar\leq\ar e^{-\frac{k}{\theta}}\mathbb{E}_{\nu}\Bigg(\exp\bigg\{\frac{k}{\theta}\Big(\theta^{-k}u^{-1}
\sum_{i>\big\lfloor\theta^{-\frac{k}{{\be-1}}}k^{\frac{2}{{\be-1}}}\big\rfloor}\hat{\tau_i}\Big)\bigg\}\Bigg)\cr
\ar=\ar e^{-\frac{k}{\theta}}\prod_{i>\big\lfloor\theta^{-\frac{k}{{\be-1}}}k^{\frac{2}{{\be-1}}}\big\rfloor}
\mathbb{E}_{\nu}\(\exp\left\{k\theta^{-\(k+1\)}u^{-1}\hat{\tau}_i\right\}\).
\eeqlb
Condition \ref{con2} implies that for $k$ large enough, there exists a constant $C$ such that
\beqnn
{\lambda_{i}}^{-1}<\sum_{b>\lfloor\theta^{-\frac{k}{{\be-1}}}k^{\frac{2}{{\be-1}}}\big\rfloor}^{\infty}{\lambda_{b}}^{-1}<C{\Big(\Big\lfloor\theta^{-\frac{k}{{\be-1}}}k^{\frac{2}{{\be-1}}}\Big\rfloor\Big)}^{-{(\be-1)}}
\eeqnn
holds for any $i>\big\lfloor\theta^{-\frac{k}{{\be-1}}}k^{\frac{2}{{\be-1}}}\big\rfloor$. That is to say,
${\lambda_{i}}>C^{-1}{\theta^{-k}k^{2}}$
holds for any $i>\big\lfloor\theta^{-\frac{k}{{\be-1}}}k^{\frac{2}{{\be-1}}}\big\rfloor$.
Let's choose $k$ large enough satisfying
${\lambda_{i}}>2k\theta^{-\(k+1\)}u^{-1}$,
which is also sufficient for the existence of exponential moments in (\ref{eq:I_2}) and therefore,
\beqnn
I_{2,k}\ar\leq\ar e^{-\frac{k}{\theta}}\prod_{i>\big\lfloor\theta^{-\frac{k}{{\be-1}}}k^{\frac{2}{{\be-1}}}\big\rfloor}
\frac{\lambda_i}{\lambda_i-k\theta^{-\(k+1\)}u^{-1}}
\equiv e^{-\frac{k}{\theta}}Q,
\eeqnn
where
\beqnn
\log Q\ar=\ar \sum_{i>\big\lfloor\theta^{-\frac{k}{{\be-1}}}k^{\frac{2}{{\be-1}}}\big\rfloor}\log\bigg(1+\frac{k\theta^{-\(k+1\)}u^{-1}}{\lambda_i-k\theta^{-\(k+1\)}u^{-1}}\bigg)\cr
\ar\leq\ar\sum_{i>\big\lfloor\theta^{-\frac{k}{{\be-1}}}k^{\frac{2}{{\be-1}}}\big\rfloor}\frac{k\theta^{-\(k+1\)}u^{-1}}
{\lambda_i-k\theta^{-\(k+1\)}u^{-1}}\cr
\ar\leq\ar 2k\theta^{-\(k+1\)}u^{-1}\sum_{i>\big\lfloor\theta^{-\frac{k}{{\be-1}}}k^{\frac{2}{{\be-1}}}\big\rfloor}\frac{1}{\lambda_i}\cr
\ar\leq\ar2k\theta^{-\(k+1\)}u^{-1}C\(\big\lfloor\theta^{-\frac{k}{{\be-1}}}k^{\frac{2}{{\be-1}}}\big\rfloor\)^{-{(\be-1)}}\cr
\ar\leq\ar 2Cu^{-1}\theta^{-1}k^{-1}.
\eeqnn
It implies that there is a constant $C_2\(C,u,\theta\)$ depending on $C,\,u,$ and $\theta$ such that $Q\leq C_2\(C,u,\theta\)$. Then
\beqnn
I_{2,k}\leq C_2(C,u,\theta) e^{-\frac{k}{\theta}}.
\eeqnn
Sum over $k>n$ to derive
{\small
\beqnn
\ar\ar\mathbb{P}_{\nu}(B_n)\leq\sum_{k>n}\mathbb{P}_{\nu}(A_k)\cr
\ar\ar\quad\leq C_1(c,d,u,\theta)\sum_{k>n}\theta^{\(\frac{c^2}{2}-\frac{1}{{\be-1}}\)k}k^{\frac{2}{{\be-1}}}\left[\log\(\theta^{-k+1}u^{-1}(1-\theta)^{-1}\)\right]^{\frac{d}{2}-1}
+C_2(C,u,\theta) \sum_{k>n}e^{-\frac{k}{\theta}}\cr
\ar\ar\quad\leq C_1(c,d,u,\theta)\theta^{\big(\frac{c^2}{2}-\frac{1}{{\be-1}}\big)n}n^{\frac{2}{{\be-1}}+\frac{d}{2}-1}
\sum_{k>n}\frac{\theta^{\big(\frac{c^2}{2}-\frac{1}{{\be-1}}\big)(k-n)}k^{\frac{2}{{\be-1}}}
\big[(k-1)\log\frac{1}{\theta}+\log\frac{1}{u(1-\theta)}\big]^{\frac{d}{2}-1}}
{n^{\frac{2}{{\be-1}}+\frac{d}{2}-1}}\cr
\ar\ar\quad\qquad+C_2(C,u,\theta)\frac{e^{-\frac{n+1}{\theta}}}{1-e^{-\frac{1}{\theta}}}\cr
\ar\ar\quad\equiv C_1(c,d,u,\theta)\theta^{\(\frac{c^2}{2}-\frac{1}{{\be-1}}\)n}n^{\frac{2}{{\be-1}}+\frac{d}{2}-1}\al_n(c,{\be},u,\theta)
+\bar{C}(C,u,\theta)e^{-\frac{n}{\theta}}.
\eeqnn
}
It follows from the Dominated Convergence Theorem that
\beqnn
\lim_{n\rightarrow\infty}\al_n(c,{\be},u,\theta)= \(-\log{\theta}\)^{\frac{d}{2}-1}\sum_{j=1}^{\infty}\theta^{\(\frac{c^2}{2}-\frac{1}{{\be-1}}\)j}=\al(c,{\be},u,\theta).
\eeqnn
Consequently, given any $\epsilon>0$, there exists an $n_{\epsilon}=n_{\epsilon}(c,{\be},u,\theta)$ such that for any $n\geq n_{\epsilon}$,
\beqnn
\al_n(c,{\be},u,\theta)\leq (1+\epsilon)\al(c,{\be},u,\theta).
\eeqnn
We may choose $C(c,d,{\be},u,\theta)$ satisfying
\beqnn
C(c,d,{\be},u,\theta)\geq C_1(c,d,u,\theta)\times\left[(1+\epsilon)\al(c,{\be},u,\theta)\vee\max\{\al_n(c,{\be},u,\theta),n<n_{\epsilon}\}\right].
\eeqnn
As a result,
\beqnn
\mathbb{P}_{\nu}(B_n)\ar\leq\ar C(c,d,{\be},u,\theta)\theta^{\(\frac{c^2}{2}-\frac{1}{{\be-1}}\)n}n^{\frac{2}{{\be-1}}+\frac{d}{2}-1}+\bar{C}(C,u,\theta)e^{-\frac{n}{\theta}}.
\eeqnn

\ref{le:modb} Given any $\theta\in(0,1)$, $u\in(0,1]$, let
\beqlb\label{eq:gamma_n}
\gamma_n(u,\theta)\ar\equiv\ar\sum_{j=1}^{\infty}\sqrt{\theta^{j-1}\frac{(n+j-1)\log\frac{1}{\theta}
+\log\frac{1}{u(1-\theta)}}{n\log\frac{1}{\theta}+\log\frac{1}{u}}}.
\eeqlb
It also follows from the Dominated Convergence Theorem that
\beqnn
\lim_{n\rightarrow\infty}\gamma_n(u,\theta)=\sum_{j=1}^{\infty}\sqrt{\theta^{j-1}}=\big(1-\theta^{\frac12}\big)^{-1}.
\eeqnn
Therefore, $\forall$ $\epsilon>0$, there is an $n_0=n_0(\epsilon,u,\theta)$ such that for any $n\geq n_0$,
\beqlb\label{eq:limgamma}
\gamma_n(u,\theta)\leq (1+\epsilon)\big(1-\theta^{\frac12}\big)^{-1}.
\eeqlb
By the lookdown representation of Fleming-Viot process, it is easy to see that
\beqnn
H^t\({t-\theta^nu,t}\)\leq\sum_{k> n}H^t\({t-\theta^{k-1}u,t-\theta^{k}u}\).
\eeqnn
Then for any $n\geq n_0$, $t\geq \theta^n u$, if $B_n$ does not occur, we have
\beqnn
H^t\({t-\theta^nu,t}\)
\ar\leq\ar\sum_{k>n}ch\big(\theta^{k-1}u(1-\theta)\big)\cr
\ar\leq\ar\sum_{k>n}c\sqrt{\theta^{k-1}u\(1-\theta\)\log\big(\theta^{-k+1}u^{-1}\(1-\theta\)^{-1}\big)}\cr
\ar=\ar c\sqrt{\theta^{n}u\log\(\theta^{-n}u^{-1}\)}
\sum_{k>n}\sqrt{\theta^{k-n-1}\(1-\theta\)\frac{(k-1)\log\frac{1}{\theta}+\log\frac{1}{u(1-\theta)}}{n\log\frac{1}{\theta}+\log\frac{1}{u}}}\cr
\ar=\ar ch(\theta^nu)\sqrt{1-\theta}\gamma_n(u,\theta)\cr
\ar\leq\ar c(1+\epsilon)(1-\theta)^{\frac12}(1-\theta^{\frac12})^{-1}h(\theta^nu).
\eeqnn
The proof is complete.
\end{proof}

\subsubsection{Upper bounds for the global modulus of continuity}
For the upper bounds we refine  the approach in \cite{LZ2}. To obtain a more precise upper bound for the global modulus of continuity, we adapt the scheme for partitioning the time interval in \cite{DIP} to work with a finer collection of grid points. Let $\theta$ be small enough and $M$ be large enough. Divide $[0,T]$ into subintervals of length $\theta^n$ and then further divide each subinterval  equally  into $M$ smaller intervals. All the partition points constitute the collection of grid points. Obviously the distance between two adjacent grid points is $\theta^n/M$. The maximal dislocation between any two close grid points with possible spacing of $j\theta^n/M$, $j=1,\ldots,M$ can be obtained through Lemma \ref{le:mod} \ref{le:modb} by  setting $u=j/M$. As in L\'evy's proof for Brownian motion, see for example \cite[Theorem 4.5]{DIP}, we can get a better approximation for any two close times $s<t\in[0,T]$ with $\theta^{n+1}<t-s\leq\theta^n$
by grid points $s_0$ $(s<s_0)$ and $t_0$ $(t_0<t)$.
The maximal dislocation between $s$ and $s_0$ or $t_0$ and $t$ is estimated by the previously known result in  Lemma \ref{le:mod1} and bounded by $ C^{*}h(\theta^n/M)$, while the maximal dislocation between $s_0$ and $t_0$ can be obtained by the maximal dislocation between two  neighboring grid points. The upper bound for the maximal dislocation between $s$ and $t$ is subsequently reached by taking a summation.

\begin{proposition}\label{le:upp}
Under Condition \ref{con2},
if $c>\sqrt{2\be/(\be-1)}$, then for any $\nu\in M_1(\mathbb{R}^d)$, $\mathbb{P}_{\nu}$-a.s. there exists a positive random variable $\delta$ such that for any $0\leq s<t\leq T$ with $t-s\leq\delta$,
\beqnn
H^t\({s,t}\)\leq ch(t-s).
\eeqnn
Moreover, there are positive constants $\tilde{c}_1$, $\tilde{c}_2$ and $\tilde{c}_3$ such that
\beqnn
\mathbb{P}_\nu\(\delta\leq\rho\)\leq \tilde{c}_1\rho^{\tilde{c}_2}\text{\,\,\,\,\,\,\,\,if\,\,}0\leq\rho\leq\tilde{c}_3.
\eeqnn
\end{proposition}

\begin{proof}
For any $c>\sqrt{2\be/(\be-1)}$, choose $\bar{c}\in(\sqrt{2\be/(\be-1)},c)$ and  $\epsilon,\,\theta\in(0,1)$ small enough so that
\beqnn
\bar{c}(1+\epsilon)(1-\theta)^{\frac12}(1-\theta^{\frac12})^{-1}<c.
\eeqnn
Choose $M$ large enough such that $\theta M>2$ and
\beqlb\label{eq:con3}
2C^{*}\sqrt{\frac{\log M}{M\theta\log({1}/{\theta})}}+\bar{c}(1+\epsilon)(1-\theta)^{\frac12}(1-\theta^{\frac12})^{-1}<c,
\eeqlb
where $C^{*}$ is the constant in (\ref{eq:mod1}). Let
\beqlb\label{eq:n_0}
\bar{n}_0=\bar{n}_0\(\epsilon,M,\theta\)=\max\Big\{n_0\big(\epsilon,{j}/{M},\theta\big):j=1,\ldots,M\Big\},
\eeqlb
where $n_0\(\cdot,\cdot,\cdot\)$ is given in Lemma \ref{le:mod} \ref{le:modb}.
Then for any $n\geq \bar{n}_0$, 
Lemma \ref{le:mod} \ref{le:modb} can be applied to any two grid points with possible spacing \{$j\theta^n/M$, $j=1,\ldots,M$\}.
Define
\beqnn
D_n\ar:=\ar\bigcup_{k=n}^{\infty}\bigcup_{0\leq m\leq T\theta^{-k}}\bigcup_{i=0}^{M-1}\bigcup_{j=1}^{M}
B_k\Big(\big(m+\frac{i}{M}\big)\theta^k,\frac{j}{M},\theta,\bar{c}\Big).
\eeqnn
Set
\beqnn
J\equiv J\(\epsilon,M,\theta,T\):=\min\left\{n: D_n\,\text{does not occur}\right\}\vee \bar{n}_0.
\eeqnn
If $n>\bar{n}_0$, Lemma \ref{le:mod} \ref{le:moda} yields
{\small
\beqlb\label{eq:J}
\ar\ar\mathbb{P}_{\nu}\(J\geq n\)\,\leq\,\mathbb{P}_{\nu}\(D_{n-1}\)\cr
\ar\ar\quad\leq\sum_{k=n-1}^{\infty}\big(T\theta^{-k}+1\big)M^2\max\Big\{C\big(\bar{c},d,\be,\frac{j}{M},\theta\big): j=1,\ldots,M\Big\}\theta^{\big(\frac{\bar{c}^2}{2}-\frac{1}{{\be-1}}\big)k}k^{\frac{2}{{\be-1}}+\frac{d}{2}-1}\cr
\ar\ar\qquad+\sum_{k=n-1}^{\infty}\big(T\theta^{-k}+1\big)M^2\max\Big\{\bar{C}\big(C,\frac{j}{M},\theta\big): j=1,\ldots,M\Big\}e^{-\frac{k}{\theta}}\cr
\ar\ar \quad\leq\hat{C}\(\bar{c},C,d,\be,M,\theta,T\)\Big[\theta^{\(\frac{\bar{c}^2}{2}-\frac{1}{{\be-1}}-1\)n}n^{\frac{2}{{\be-1}}+\frac{d}{2}-1}\vee\big(\theta e^{\frac{1}{\theta}}\big)^{-n}\Big].
\eeqlb
}
Equation (\ref{eq:J}) implies that
{\small\beqnn
\ar\ar\log\mathbb{P}_{\nu}\(\theta^J\leq\rho\)\,=\,\log\mathbb{P}_{\nu}\(J\geq\log\rho/\log\theta\)\cr
\ar\ar\quad\leq\log\hat{C}\(\bar{c},C,d,\be,M,\theta,T\)\cr
\ar\ar\qquad+\Bigg[\bigg(\Big(\frac{\bar{c}^2}{2}-\frac{1}{{\be-1}}-1\Big)\log\rho
+\Big(\frac{2}{{\be-1}}+\frac{d}{2}-1\Big)\log\Big(\frac{\log\rho}{\log\theta}\Big)\bigg)\bigvee \bigg(-\frac{\log\rho}{\log\theta}\log\Big(\theta e^{\frac{1}{\theta}}\Big)\bigg)\Bigg].
\eeqnn}
Let $\delta:=\theta^{*}\wedge \theta^J$ with $\theta^{*}$  given in Lemma \ref{le:mod1} satisfying $\theta^{*}\leq e^{-1}$. Combining this result with (\ref{eq:theta}), there are positive constants $\tilde{c}_1$, $\tilde{c}_2$ and $\tilde{c}_3$ such that
\beqnn
\mathbb{P}_{\nu}\(\delta\leq\rho\)\leq\mathbb{P}_{\nu}(\theta^{*}\leq\rho)+\mathbb{P}_{\nu}(\theta^{J}\leq\rho)\leq \tilde{c}_1\rho^{\tilde{c}_2}\text{\,\,\,\,\,\,\,\,if\,\,}0\leq\rho\leq\tilde{c}_3.
\eeqnn

For any $0\leq s< t\leq T$ with $t-s\leq\delta$, choose $n$ such that
$\theta^{n+1}<t-s\leq \theta^{n}.$
We clarify that here the value of $n$ is greater than or equal to $J$ since $\delta\leq \theta^{J}$ and so $D_n$ does not occur.
Then $s$ and $t$ are approximated by grid points as follows. Choose
$0\leq m\leq T\theta^{-n}$, $i\in\{0,\ldots,M-1\}$ and  $j\in\{1,\ldots,M\}$ such that
\beqnn
t_0\ar=\ar\(m+(i/M)\)\theta^n\leq t<\(m+((i+1)/M)\)\theta^n
\eeqnn
and
\beqnn
\(m+(i-j-1)/M\)\theta^n<s\leq \(m+((i-j)/M)\)\theta^n=s_0<t_0,
\eeqnn
where the existence of $j$ is due to  the fact that $\theta^{n+1}>2\theta^n/M$.
Note that
\beqnn
(s_0-s)\vee (t-t_0)\leq \theta^n/M.
\eeqnn
The upper bound for the maximal dislocation between $s$ and $s_0$ or between $t_0$ and $t$ is derived by
the modulus of continuity result in Lemma \ref{le:mod1} and consequently,
\beqlb\label{eq:mod2}
H^{s_0}\({s,s_0}\)\vee H^{t}\({t_0,t}\)\ar\leq\ar C^{*}h\(\theta^n/M\)\cr
\ar=\ar C^{*}h\(\theta^{n+1}\)\sqrt{\frac{1}{\theta M}\frac{\log M+n\log\({1}/{\theta}\)}{\log\({1}/{\theta}\)+n\log\({1}/{\theta}\)}}\cr
\ar\leq\ar C^{*}h\(\theta^{n+1}\)\sqrt{\frac{\log(M)}{\theta M\log(1/\theta)}}\cr
\ar\leq\ar C^{*}\sqrt{\frac{\log(M)}{\theta M\log(1/\theta)}}h\(t-s\),
\eeqlb
where the last line is due to the monotonicity of $h(x)$ on $(0,e^{-1}]$.
Recall that $B_n(t,u,\theta,c)$ is given by (\ref{eq:B_n}). We set $t=t_0$, $u=j/M$, and $c=\bar{c}$ in (\ref{eq:B_n}). Since  $s_0=t_0-(j/M)\theta^n$ and the value of $n$ is large enough such that $B_n(t_0,j/M,\theta,\bar{c})$ does not occur, then the maximal dislocation between grid points $s_0$ and $t_0$
follows from Lemma \ref{le:mod} \ref{le:modb}. Consequently,
\beqlb\label{eq:grid}
H^{t_0}\({s_0,t_0}\)\ar\leq\ar\bar{c}(1+\epsilon)(1-\theta)^{\frac12}(1-\theta^{\frac12})^{-1}h\(t_0-s_0\)\cr
\ar\leq\ar\bar{c}(1+\epsilon)(1-\theta)^{\frac12}(1-\theta^{\frac12})^{-1}h\(t-s\),
\eeqlb
where the last line also follows from the monotonicity of $h(x)$. The lookdown representation implies that
\beqnn
H^t\({s,t}\)\leq H^{s_0}\({s,s_0}\)+H^{t_0}\({s_0,t_0}\)+H^t\({t_0,t}\).
\eeqnn
Combining with (\ref{eq:con3}), (\ref{eq:mod2}) and (\ref{eq:grid}), we have
\beqnn
H^t\({s,t}\)\ar\leq \ar\bigg[2C^{*}\sqrt{\frac{\log(M)}{\theta M\log(1/\theta)}}+\bar{c}(1+\epsilon)(1-\theta)^{\frac12}(1-\theta^{\frac12})^{-1}\bigg]h\(t-s\)\cr
\ar<\ar ch\(t-s\).
\eeqnn
The proof is complete.
\end{proof}

Combining with the result in Proposition \ref{le:upp},
we can easily obtain the upper bound of global modulus of continuity for the ancestry process as follows.
\begin{lemma}\label{th:global1}
Assume that Condition \ref{con2} holds. Then for any $T>0$ and any $\nu\in M_1(\mathbb{R}^d)$, we have $\mathbb{P}_\nu$-a.s.
		\begin{equation*}\label{uniform_upper}
	\limsup_{\ep\to 0+}\sup_{\ep\leq t\le T }\frac{H^t\({t-\ep,t}\)}{h(\ep)}\leq\sqrt{\frac{2\beta}{\beta-1}}.
\end{equation*}

\end{lemma}

\subsubsection{Lower bounds for the global modulus of continuity}
\begin{lemma}\label{th:global2}
Assume that either  $\La(\{0\})>0$ or both $\La(\{0\})=0$ and Condition {\ref{eq:laup}} hold for $1<\beta<2$.
Then for any $T>0$ and any $\nu\in M_1(\mathbb{R}^d)$, we have $\mathbb{P}_\nu$-a.s.
	\begin{equation*}\label{uniform_upper}
		\limsup_{\ep\to 0+}\sup_{\ep\leq t\le T }\frac{H^t\({t-\ep,t}\)}{h(\ep)}\geq\sqrt{\frac{2\beta}{\beta-1}}.
	\end{equation*}

\end{lemma}

\begin{proof}
To show the lower bound, we only need to find a case in which the lower bound is attained. Let's divide the time interval $[0,T]$ into  subintervals of equal length $2^{-n}$. 
For any $\be\in(1,2]$ and $0<c<\sqrt{{2\be}/{(\be-1)}}$,  choose a constant $\al^{*}$ satisfying
\beqlb\label{eq:a*}
\al^{*}\in\(0,\({\be}/({2(\be-1)})-{c^2}/{4}\)\wedge {1}/{2}\).
\eeqlb
Let $ C(\be,A,\ep,c,\al^{*})$ be a constant such that
\beqlb\label{eq:C}
C(\be,A,\ep,c,\al^{*})\leq\begin{cases} e^{-24\times 2^{-\al^{*}}}C(A,\ep,\be),\text{\,\,\,\,if\,\,\,\,}\be\in(1,2),\cr
e^{-24\times 2^{-\al^{*}}}\times 2\(\La([0,1])\)^{-1},\text{\,\,\,\,if\,\,\,\,}\be=2,
\end{cases}
\eeqlb
where $C(A,\ep,\be)$ is the constant from (\ref{eq:v1}).
For notational simplicity, write
\beqnn
E_k\ar:=\ar\left\{ N\((k-1)2^{-n},k2^{-n}\)\geq C(\beta,A,\ep,c,\al^{*})2^{\frac{n}{\beta-1}}\right\}
\eeqnn
with $k=1,2,\ldots,T2^n$. 
Proposition \ref{pro:ancestor} implies that for $n$ large enough,
\beqnn
\mathbb{P}_{\nu}\(N\((k-1)2^{-n},k2^{-n}\)\geq e^{-24\times 2^{-n\al^{*}}}v(2^{-n})\)\geq 1-O\(2^{-n(1-2\al^{*})}\).
\eeqnn
By  Proposition \ref{pro:v},
\beqnn
e^{-24\times 2^{-n\al^{*}}}v(2^{-n})\geq C(\be,A,\ep,c,\al^{*})2^{\frac{n}{\be-1}}
\eeqnn
holds for $n$ large enough. Therefore,
\beqnn
\mathbb{P}_{\nu}\(E_k\)\geq 1-O\(2^{-n(1-2\al^{*})}\)
\eeqnn
holds for $n$ large enough. By the lookdown representation,  over
nonoverlapping time intervals the maximal dislocations of the individuals from their respective ancestors  are independent. We then  have
\beqnn
\mathbb{P}_{\nu}\({F_n^{c}}\)\ar=\ar\prod_{1\leq k\leq T2^n}\mathbb{P}_{\nu}\(H^{k2^{-n}}\({\(k-1\)2^{-n},k2^{-n}}\)\leq ch\(2^{-n}\)\)\cr
\ar\leq\ar\prod_{1\leq k\leq T2^n}\bigg[\mathbb{P}_{\nu}\(H^{k2^{-n}}\({\(k-1\)2^{-n},k2^{-n}}\)\leq ch\(2^{-n}\),E_k\)+\mathbb{P}\( E_k^c\)\bigg]\cr
\ar\leq\ar\prod_{1\leq k\leq T2^n}\bigg[\bigg(\mathbb{P}\(|B(1)|\leq c\sqrt{n\log2}\)\bigg)^{\big\lfloor { C(\beta, A,\ep,c,\al^{*})}2^{\frac{n}{\beta-1}}\big\rfloor}+O\(2^{-n(1-2\al^{*})}\)\bigg]\cr
\ar\leq\ar\Bigg[\bigg(1-\mathbbm{c}_1(d)c^{d-2}(n\log2)^{\frac{d}{2}-1}2^{-\frac{c^2n}{2}}\bigg)
^{\big\lfloor  C(\beta,A,\ep,c,\al^{*})2^{\frac{n}{\beta-1}}\big\rfloor}+O\(2^{-n(1-2\al^{*})}\)\Bigg]^{T2^n},
\eeqnn
where we have used  (\ref{eq:bm}) in  the last line. Therefore,
\beqnn
\ar\ar\lim_{n\rightarrow\infty}\mathbb{P}_{\nu}\({F_n^{c}}\)\cr
\ar\ar\quad\leq\lim_{n\rightarrow\infty}
\Bigg[\bigg(1-\mathbbm{c}_1(d)c^{d-2}(n\log2)^{\frac{d}{2}-1}2^{-\frac{c^2n}{2}}\bigg)
^{C(\beta,A,\ep,c,\al^{*})2^{\frac{n}{\beta-1}}}+O\(2^{-n(1-2\al^{*})}\)\Bigg]^{T2^n}.
\eeqnn

In case of $c<\sqrt{\frac{2}{\beta-1}}$, the above limit has the form of $0^{\infty}$ and it is exactly equal to $0$.

For $\sqrt{\frac{2}{\beta-1}}\leq c<\sqrt{2\(1+\frac{1}{\beta-1}\)}=\sqrt{\frac{2\beta}{\beta-1}},$
 we proceed to show that the limit is still $0$.
By Taylor's formula,
\beqnn
\ar\ar\log\bigg[\bigg(1-\mathbbm{c}_1(d)c^{d-2}(n\log2)^{\frac{d}{2}-1}2^{-\frac{c^2n}{2}}\bigg)^{ C(\beta, A,\ep,c,\al^{*})2^{\frac{n}{\beta-1}}}\bigg]\cr
\ar\ar\quad=- C(\beta, A,\ep,c,\al^{*})2^{\frac{n}{\beta-1}}\bigg[\mathbbm{c}_1(d)c^{d-2}(n\log2)^{\frac{d}{2}-1}2^{-\frac{c^2n}{2}}+O\Big(n^{d-2}2^{-c^2n}\Big)\bigg].
\eeqnn
Then
\beqnn
\lim_{n\rightarrow\infty}\mathbb{P}_{\nu}\({F_n^{c}}\)\ar\leq\ar\lim_{n\rightarrow\infty}
\bigg[\exp\bigg\{-{C(\beta, A,\ep,c,d,\al^{*})}n^{\frac{d}{2}-1}2^{-n\big(\frac{c^2}{2}-\frac{1}{\beta-1}\big)}\bigg\}
+O\(2^{-n(1-2\al^{*})}\)\bigg]^{T2^n}.
\eeqnn
If $c=\sqrt{\frac{2}{\beta-1}}$, the dominant term in the above braces is $n^{\frac{d}{2}-1}$. If $d= 2$, the limit has the form of $q^{\infty}$ with $0<q<1$ and consequently, it is equal to $0$. If $d<2$, then $d/2-1<0$ and the limit equals to $\lim_{n\rightarrow\infty}[1-{C(\beta, A,\ep,c,d,\al^{*})}n^{\frac{d}{2}-1}]^{T2^n}=0$. If $d> 2$,  then $d/2-1>0$ and the term in the braces tends to $-\infty$ as $n\rightarrow\infty$. Therefore, the limit has the form of $0^{\infty}$ and  it is also equal to $0$.
If $\sqrt{\frac{2}{\beta-1}}<c<\sqrt{\frac{2\beta}{\beta-1}}$, we have
$0<\frac{c^2}{2}-\frac{1}{\beta-1}<1$. This fact, combined with (\ref{eq:a*}) gives
$\frac{c^2}{2}-\frac{1}{\beta-1}\,<\, 1-2\al^{*}.$
Therefore,
\beqnn
\lim_{n\rightarrow\infty}\mathbb{P}_{\nu}\({F_n^{c}}\)\ar\leq\ar\lim_{n\rightarrow\infty}
\left[1-C(\beta, A,\ep,c,d,\al^{*})n^{\frac{d}{2}-1}2^{-n\big(\frac{c^2}{2}-\frac{1}{\beta-1}\big)}\right]^{T2^n}\,=\,0.
\eeqnn
The proof is complete. 
\end{proof}

\subsubsection{Upper bounds for the local left modulus of continuity}
The local left modulus of continuity for $\La$-Fleming-Viot ancestry process follows from an argument that is along the same line of the global modulus of continuity, but  is simpler. Instead of considering the grid points on the whole interval, we only need to carry out the partitions on a left neighborhood of a fixed time.
\begin{proposition}\label{le:localupp}
Under Condition \ref{con2},
if $c>\sqrt{2/(\be-1)}$, then for each fixed $t>0$  and any $\nu\in M_1(\mathbb{R}^d)$, $\mathbb{P}_{\nu}$-a.s. there exists a positive random variable $\delta$ such that for all $s$ with $0< t-s\leq \delta$,
\beqnn
H^t\({s,t}\)\leq ch(t-s).
\eeqnn
Moreover, there are positive constants $\tilde{c}_4$, $\tilde{c}_5$ and $\tilde{c}_6$ such that
\beqnn
\mathbb{P}_{\nu}\(\delta\leq\rho\)\leq \tilde{c}_4\rho^{\tilde{c}_5}\text{\,\,\,\,\,\,\,\,if\,\,}0\leq\rho\leq\tilde{c}_6.
\eeqnn
\end{proposition}

\begin{proof}
For any $c>\sqrt{2/(\be-1)}$, let $\bar{c}\in(\sqrt{2/(\be-1)},c)$ and choose $\epsilon,\,\theta\in(0,1)$ small enough so that
\beqnn
\bar{c}(1+\epsilon)(1-\theta)^{\frac12}(1-\theta^{\frac12})^{-1}<c.
\eeqnn
Then choose $M$ large enough such that $\theta M>2$ and
\beqlb\label{eq:localcon3}
C^{*}\sqrt{\frac{\log M}{M\theta\log({1}/{\theta})}}+\bar{c}(1+\epsilon)(1-\theta)^{\frac12}(1-\theta^{\frac12})^{-1}<c,
\eeqlb
where $C^{*}$ is the constant in (\ref{eq:mod1}).
For each fixed $t>0$, define
\beqnn
\ar\ar D^{*}_{n,t}:=\bigcup_{k=n}^{\infty}\bigcup_{i=0}^{M-1}\bigcup_{j=1}^{M-i}
B_k\Big(t-\frac{i}{M}\theta^k,\frac{j}{M},\theta,\bar{c}\Big).
\eeqnn
Set
\beqnn
\bar{J}\equiv\bar{J}\(\epsilon,M,\theta,t\):=\min\left\{n: D^{*}_{n,t}\,\text{does not occur}\right\}\vee \bar{n}_0,
\eeqnn
where $\bar{n}_0$ is defined by (\ref{eq:n_0}).
If $n>\bar{n}_0$, Lemma \ref{le:mod} \ref{le:moda} implies
\beqlb\label{eq:barJ}
\ar\ar\mathbb{P}_{\nu}\(\bar{J}\geq n\)\,\leq\,\mathbb{P}_{\nu}\(D^{*}_{n-1,t}\)\cr
\ar\ar\quad\leq\sum_{k=n-1}^{\infty}M^2\max\Big\{C\big(\bar{c},d,\be,\frac{j}{M},\theta\big): j=1,\ldots,M\Big\}\theta^{\big(\frac{\bar{c}^2}{2}-\frac{1}{{\be-1}}\big)k}k^{\frac{2}{{\be-1}}+\frac{d}{2}-1}\cr
\ar\ar\qquad+\sum_{k=n-1}^{\infty}M^2\max\Big\{\bar{C}\big(C,\frac{j}{M},\theta\big): j=1,\ldots,M\Big\}e^{-\frac{k}{\theta}}\cr
\ar\ar \quad\leq\hat{C}\(\bar{c},C,d,\be,M,\theta\)\Big[\theta^{\(\frac{\bar{c}^2}{2}-\frac{1}{{\be-1}}\)n}n^{\frac{2}{{\be-1}}+\frac{d}{2}-1}\vee\big( e^{-\frac{n}{\theta}}\big)\Big].
\eeqlb
Let $\delta:=\theta^{*}\wedge \theta^{\bar{J}}$ with $\theta^{*}$  given in Lemma \ref{le:mod1} satisfying $\theta^{*}\leq e^{-1}$. Equation (\ref{eq:barJ}) yields
{\small
\beqnn
\ar\ar\log\mathbb{P}_{\nu}\(\theta^{\bar{J}}\leq\rho\)\,=\,\log\mathbb{P}_{\nu}\(\bar{J}\geq\log\rho/\log\theta\)\cr
\ar\ar\quad\leq\log\hat{C}\(\bar{c},C,d,\be,M,\theta\)\cr
\ar\ar\qquad+\Bigg[\bigg(\Big(\frac{\bar{c}^2}{2}-\frac{1}{{\be-1}}\Big)\log\rho
+\Big(\frac{2}{{\be-1}}+\frac{d}{2}-1\Big)\log\Big(\frac{\log\rho}{\log\theta}\Big)\bigg)\bigvee \bigg(-\frac{\log\rho}{\log\theta}\log\Big( e^{\frac{1}{\theta}}\Big)\bigg)\Bigg].
\eeqnn
}
Combining with (\ref{eq:theta}), there are positive constants $\tilde{c}_4$, $\tilde{c}_5$ and $\tilde{c}_6$ such that
\beqlb\label{eq:delta}
\mathbb{P}_{\nu}\(\delta\leq\rho\)\leq\mathbb{P}_{\nu}(\theta^{*}\leq\rho)+\mathbb{P}_{\nu}(\theta^{\bar{J}}\leq\rho)\leq \tilde{c}_4\rho^{\tilde{c}_5}\text{\,\,\,\,\,\,\,\,if\,\,}0\leq\rho\leq\tilde{c}_6.
\eeqlb
Given fixed $t>0$, for all $s$ with $0<t-s\leq\delta$, choose $n$ such that
$\theta^{n+1}<t-s\leq \theta^{n}.$
Note that the value of $n$ is greater than or equal to $\bar{J}$ since $\delta\leq \theta^{\bar{J}}$ and so $D^{*}_{n,t}$ does not occur.
Then $s$ is approximated by grid points. Choose
$i\in\{1,\ldots,M\}$  such that
\beqnn
t-\((i+1)/M\)\theta^n<s\leq t-\(i/M\)\theta^n=s_0.
\eeqnn
The upper bound for the maximal dislocation between $s$ and $s_0$ is derived by
Lemma \ref{le:mod1}. Consequently, as in the derivation of (\ref{eq:mod2}), we obtain
\beqlb\label{eq:localmod2}
H^{s_0}\({s,s_0}\)
\ar\leq\ar C^{*}\sqrt{\frac{\log(M)}{\theta M\log(1/\theta)}}h\(t-s\).
\eeqlb
The maximal dislocation between grid points $s_0$ and $t$
follows from Lemma \ref{le:mod} \ref{le:modb} by setting $u=i/M$ so that
\beqlb\label{eq:localt}
H^{t}\({s_0,t}\)\ar\leq\ar\bar{c}(1+\epsilon)(1-\theta)^{\frac12}(1-\theta^{\frac12})^{-1}h\(t-s_0\)\cr
\ar\leq\ar\bar{c}(1+\epsilon)(1-\theta)^{\frac12}(1-\theta^{\frac12})^{-1}h\(t-s\),
\eeqlb
where the last line follows from the monotonicity of  $h(x)$ on $(0,e^{-1}]$. The lookdown representation implies that
\beqnn
H^t\({s,t}\)\leq H^{s_0}\({s,s_0}\)+H^{t}\({s_0,t}\).
\eeqnn
Combining (\ref{eq:localcon3}), (\ref{eq:localmod2}) and (\ref{eq:localt}), we have
\beqnn
H^t\({s,t}\)\ar\leq \ar\bigg[C^{*}\sqrt{\frac{\log(M)}{\theta M\log(1/\theta)}}+\bar{c}(1+\epsilon)(1-\theta)^{\frac12}(1-\theta^{\frac12})^{-1}\bigg]h\(t-s\)
<ch\(t-s\).
\eeqnn
The proof is complete.
\end{proof}

Combining with the result in Proposition \ref{le:localupp},
we can easily obtain the upper bound of local left modulus of continuity for the ancestry process as follows.
\begin{lemma}\label{th:localleft1}
Assume that Condition \ref{con2} holds. Then for each fixed $t>0$  and any $\nu\in M_1(\mathbb{R}^d)$, we have $\mathbb{P}_\nu$-a.s.
	\begin{equation*}\label{local_lower_left}
		\limsup_{\ep\to 0+}\frac{H^t\({t-\ep,t}\)}{h(\ep)}\leq\sqrt{\frac{2}{\beta-1}}.
	\end{equation*}		
\end{lemma}

\subsubsection{Lower bounds for the local left modulus of continuity}
\begin{lemma}\label{th:localleft2}
Assume that either $\La(\{0\})>0$ or both $\La(\{0\})=0$ and  Condition {\ref{eq:laup}} hold for $1<\beta<2$.
Then for each fixed $t>0$ and any $\nu\in M_1(\mathbb{R}^d)$, we have $\mathbb{P}_\nu$-a.s.
	\begin{equation*}\label{local_lower_left}
	\limsup_{\ep\to 0+}\frac{H^t\({t-\ep,t}\)}{h(\ep)}\geq\sqrt{\frac{2}{\beta-1}}.
   \end{equation*}	
\end{lemma}

\begin{proof}
The proof is similar to that of Lemma \ref{th:global2}. The main steps are presented as follows.
Given fixed $t>0$, set
\beqnn
F_{n,t}:=\left\{ H^{t}\(t-2^{-n},t\)>ch\(2^{-n}\)\right\}.
\eeqnn
For any $\be\in(1,2]$ and $\al^{*}\in(0,1/2)$,
let $ C(\be,A,\ep,c,\al^{*})$ be the constant given by (\ref{eq:C}).
Write
\beqnn
E_{n,t}\,=\,\left\{ N\(t-2^{-n},t\)\geq C(\beta, A,\ep,c,\al^{*})2^{\frac{n}{\beta-1}}\right\}.
\eeqnn
By Proposition \ref{pro:ancestor}, we see that for $n$ large enough,
\beqnn
\mathbb{P}_{\nu}\(N\(t-2^{-n},t\)\geq e^{-24\times 2^{-n\al^{*}}}v(2^{-n})\)\geq 1-O\(2^{-n(1-2\al^{*})}\).
\eeqnn
This fact, combining with Proposition \ref{pro:v}, yields
\beqnn
\mathbb{P}_{\nu}\(E_{n,t}\)\geq 1-O\(2^{-n(1-2\al^{*})}\)
\eeqnn
for $n$ large enough. Therefore,
\beqnn
\mathbb{P}_{\nu}\({F_{n,t}^{c}}\)\ar=\ar\mathbb{P}_{\nu}\(H^{t}\({t-2^{-n},t}\)\leq ch\(2^{-n}\)\)\cr
\ar\leq\ar\mathbb{P}_{\nu}\(H^{t}\({t-2^{-n},t}\)\leq ch\(2^{-n}\),E_{n,t}\)+\mathbb{P}_{\nu}\({E_{n,t}^{c}}\)\cr
\ar\leq\ar\bigg(\mathbb{P}_{\nu}\(|B(1)|\leq c\sqrt{n\log2}\)\bigg)^{\big\lfloor { C(\beta, A,\ep,c,\al^{*})}2^{\frac{n}{\beta-1}}\big\rfloor}+O\(2^{-n(1-2\al^{*})}\)\cr
\ar\leq\ar\bigg(1-\mathbbm{c}_1(d)c^{d-2}(n\log2)^{\frac{d}{2}-1}2^{-\frac{c^2n}{2}}\bigg)
^{\big\lfloor  C(\beta, A,\ep,c,\al^{*})2^{\frac{n}{\beta-1}}\big\rfloor}+O\(2^{-n(1-2\al^{*})}\),
\eeqnn
where the last line follows from (\ref{eq:bm}). Therefore,
{\small\beqlb\label{eq:F}
\ar\ar\lim_{n\rightarrow\infty}\mathbb{P}_{\nu}\({F_{n,t}^{c}}\)\cr
\ar\ar\leq\lim_{n\rightarrow\infty}
\Big[\big(1-\mathbbm{c}_1(d)c^{d-2}(n\log2)^{\frac{d}{2}-1}2^{-\frac{c^2n}{2}}\big)
^{C(\beta,A,\ep,c,\al^{*})2^{\frac{n}{\beta-1}}}+O\big(2^{-n(1-2\al^{*})}\big)\Big].
\eeqlb}
In case of $c<\sqrt{{2}/{(\beta-1)}}$, the limit of each term in the above expression is equal to $0$. Then the result follows.
\end{proof}

\subsubsection{Upper bounds for the local right modulus of continuity}
By considering the grid points in a right neighborhood of a fixed time, one can obtain the local right modulus of continuity result.
\begin{proposition}\label{le:localupp2}
Under Condition \ref{con2},
if $c>\sqrt{2/(\be-1)}$, then for each fixed $s\geq0$ and any $\nu\in M_1(\mathbb{R}^d)$, $\mathbb{P}_{\nu}$-a.s. there exists a positive random variable $\delta$ such that for all $t$ with $0< t-s\leq \delta$,
\beqnn
H^t\({s,t}\)\leq ch(t-s).
\eeqnn
Moreover, there are positive constants $\tilde{c}_7$, $\tilde{c}_8$ and $\tilde{c}_9$ such that
\beqlb\label{eq:delta2}
\mathbb{P}_{\nu}\(\delta\leq\rho\)\leq \tilde{c}_7\rho^{\tilde{c}_8}\text{\,\,\,\,\,\,\,\,if\,\,}0\leq\rho\leq\tilde{c}_9.
\eeqlb
\end{proposition}
\begin{proof}
The proof is in analogous to that of Proposition \ref{le:localupp}. Now we carry out the division in a right neighborhood of a fixed time.
For any $c>\sqrt{2/(\be-1)}$, let $\bar{c}\in(\sqrt{2/(\be-1)},c)$ and choose $\epsilon,\,\theta\in(0,1)$ small enough so that
\beqnn
\bar{c}(1+\epsilon)(1-\theta)^{\frac12}(1-\theta^{\frac12})^{-1}<c.
\eeqnn
Also choose $M$ large enough such that $\theta M>2$ and (\ref{eq:localcon3}) holds.
For each fixed $s\geq0$,  define
\beqnn
\ar\ar\tilde{D}_{n,s}:=\bigcup_{k=n}^{\infty}\bigcup_{i=1}^{M}\bigcup_{j=1}^{i}
B_k\Big(s+\frac{i}{M}\theta^k,\frac{j}{M},\theta,\bar{c}\Big).
\eeqnn
Set
\beqnn
\tilde{J}\equiv\tilde{J}\(\epsilon,M,\theta,s\):=\min\left\{n: \tilde{D}_{n,s}\,\text{does not occur}\right\}\vee \bar{n}_0,
\eeqnn
where $\bar{n}_0$ is defined by (\ref{eq:n_0}). By the same estimation as (\ref{eq:barJ}), we obtain
\beqnn
\ar\ar\mathbb{P}_{\nu}\(\tilde{J}\geq n\)\,\leq\,\mathbb{P}_{\nu}\(\tilde{D}_{n-1,s}\)\leq\hat{C}\(\bar{c},C,d,\be,M,\theta\)\Big[\theta^{\(\frac{\bar{c}^2}{2}-\frac{1}{{\be-1}}\)n}n^{\frac{2}{{\be-1}}+\frac{d}{2}-1}\vee\big( e^{-\frac{n}{\theta}}\big)\Big].
\eeqnn
Choose $\delta:=\theta^{*}\wedge \theta^{\tilde{J}}$. Similarly as (\ref{eq:delta}),  (\ref{eq:delta2}) is attained. Further, for any $t$ with $0 < t-s\leq \delta$,
there exists an $n$ such that
$\theta^{n+1}< t-s\leq \theta^{n}\text{\,\,\,\,and\,\,\,\,} \tilde{J}\geq n.$
Then $t$ is approximated by grid points  by choosing
$i\in\{1,\ldots,M\}$  such that
\beqnn
s+\(i/M\)\theta^n=t_0\leq t< s+\((i+1)/M\)\theta^n.
\eeqnn
Similar to (\ref{eq:localmod2}) and (\ref{eq:localt}), we have
\beqnn
H^{t}\({t_0,t}\)
\ar\leq\ar C^{*}\sqrt{\frac{\log(M)}{\theta M\log(1/\theta)}}h\(t-s\)
\eeqnn
and
\beqnn
H^{t_0}\({s,t_0}\)\ar\leq\ar\bar{c}(1+\epsilon)(1-\theta)^{\frac12}(1-\theta^{\frac12})^{-1}h\(t-s\).
\eeqnn
Then the result follows from $H^t\({s,t}\)\leq H^{t_0}\({s,t_0}\)+H^{t}\({t_0,t}\)$. 
\end{proof}

Combining with the result in Proposition \ref{le:localupp2},
we can easily obtain the upper bound of local right modulus of continuity for the ancestry process as follows.
\begin{lemma}\label{th:localright1}
Assume that Condition \ref{con2} holds. Then for each fixed $s\geq 0$  and any $\nu\in M_1(\mathbb{R}^d)$, we have $\mathbb{P}_\nu$-a.s.
	\begin{equation*}\label{local_lower_left}
		\limsup_{\ep\to 0+}\frac{H^{s+\ep}\({s, s+\ep}\)}{h(\ep)}\leq\sqrt{\frac{2}{\beta-1}}.
	\end{equation*}	
\end{lemma}

\subsubsection{Lower bound for the local right modulus of continuity}
\begin{lemma}\label{th:localright2}
Assume that either $\La(\{0\})>0$ or both $\La(\{0\})=0$ and  Condition {\ref{eq:laup}} hold for $1<\beta<2$.
		Then for each fixed $s\geq0$ and any $\nu\in M_1(\mathbb{R}^d)$, we have $\mathbb{P}_\nu$-a.s.
	\begin{equation*}\label{local_lower_left}
		\limsup_{\ep\to 0+}\frac{H^{s+\ep}\({s, s+\ep}\)}{h(\ep)}\geq\sqrt{\frac{2}{\beta-1}}.
	\end{equation*}	
\end{lemma}
\begin{proof}
The proof is analogous to that of Lemma \ref{th:localleft2}. We only briefly list the main steps.
For each fixed $s\geq0$, set
\beqnn
\tilde{F}_{n,s}=\left\{ H^{s+2^{-n}}\(s,s+2^{-n}\)>ch\(2^{-n}\)\right\}.
\eeqnn
Write
\beqnn
\tilde{E}_{n,s}\ar=\ar\left\{ N\(s,s+2^{-n}\)\geq C(\beta,A,\ep,c,\al^{*})2^{\frac{n}{\beta-1}}\right\},
\eeqnn
where $ C(\beta,A,\ep,c,\al^{*})$ is given by (\ref{eq:C}).
Then
\beqnn
\mathbb{P}_{\nu}\({\tilde{F}_{n,s}^{c}}\)\ar=\ar\mathbb{P}_{\nu}\(H^{s+2^{-n}}\({s,s+2^{-n}}\)\leq ch\(2^{-n}\)\)\cr
\ar\leq\ar\mathbb{P}_{\nu}\(H^{s+2^{-n}}\({s,s+2^{-n}}\)\leq ch\(2^{-n}\),\tilde{E}_{n,s}\)+\mathbb{P}_{\nu}\({\tilde{E}_{n,s}^{c}}\).
\eeqnn
The limits of these two terms in the above line can be estimated in the same way as (\ref{eq:F})  so as to derive
$\lim_{n\rightarrow\infty}\mathbb{P}_{\nu}\({\tilde{F}_{n,s}^{c}}\)=0$. The proof is complete.
\end{proof}

\subsubsection{Proofs of Theorems \ref{co:moduglobal} and \ref{co:modu}}
\begin{proof}[Proof of Theorems \ref{co:moduglobal} and \ref{co:modu}]
Combined with Remark \ref{con:compare}, Theorem \ref{co:moduglobal}   follows from Lemmas \ref{th:global1}-\ref{th:global2} and Theorem \ref{co:modu} follows from Lemmas \ref{th:localleft1}-\ref{th:localleft2} and \ref{th:localright1}-\ref{th:localright2}.
\end{proof}

\subsection{Modulus of continuities for the $\La$-Fleming-Viot support process}
For any $t\geq 0$, recall that $X(t)=\lim_{n\rightarrow\infty}\frac{1}{n}\sum_{i=1}^{n}\delta_{X_i(t)}$ with $\{X_1(t),X_2(t),\ldots\}$ being the collection of spatial locations of particles in the lookdown representation
at time $t$.
It follows from \cite[Lemma 3.1]{LZ2} that for each fixed $t>0$ and any $\nu\in M_1(\mathbb{R}^d)$,
\beqnn
\mathbb{P}_{\nu}\(\left\{X_1(t),X_2(t),\ldots\right\}\subseteq S\(X(t)\)\)=1.
\eeqnn
Adapting a proof for the corresponding result of super-Brownian motion in \cite[ Lemma III.1.6]{perkins}, the above result uniformly   holds for all $t$.
\begin{lemma}\label{le:ut}
Given $T>0$ and any $\nu\in M_1(\mathbb{R}^d)$, we have $\mathbb{P}_{\nu}$-a.s. $\{X_i(t), i=1,2,\ldots\}\subseteq S(X(t))$ for all $0\leq t\leq T$.
\end{lemma}
\begin{proof}
Write $\mathbb{Q}$ for the set of positive rational numbers. From the lookdown representation, $\mathbb{P}_{\nu}$-a.s. for all $r\in \mathbb{Q}$, \beqlb\label{eq:sup}
\{X_i(r), i=1,2,\ldots\}\subseteq S(X(r)).
\eeqlb	
Observe that there are countably many Brownian motions involved in the lookdown representation. We can thus choose a version so that they are all continuous satisfying (\ref{eq:sup}).

For any rational open ball $B\equiv B(x,\ep)$ with rational $x$ and $\ep$, write $B'\equiv B(x,\ep/2)$. For all $t\geq 0$, let $T_t(B):=\inf\{s\geq t: X(s)(B)=0\}$. Then by the strong Markov property for $X$ and by the modulus of  continuity for the ancestry process, with probability one for any $r$ and rational ball $B$, there exists a random variable $\eta>0$ such that
	\begin{equation}\label{rational}
	X(s)(B')=0\quad\text{ for all}\quad s\in [T_r(B), T_r(B)+\eta)\quad\text{ for}\quad T_r(B)<\infty.
	\end{equation}
	
We claim that on the event that all of the above hold, for any $0< s\leq T$ and $i$, $X_i(s)\in S(X(s))$. Otherwise, there exist $s, i$  such that $X_i(s)\not\in S(X(s))$. Then one can select rational balls $B$ and $B'$ satisfying $X_i(s)\in B'$ and $S(X(s))\cap B=\emptyset$.
Further choose a sequence of positive rationals $(r_n)$ with  $ r_n\uparrow s$. Then for all $n$, $T_{r_n}(B)\leq s$ and $X(t)(B')=0$ for all $t\in [T_{r_n}(B), T_{r_n}(B)+\eta_n)$ for some $\eta_n>0$, which combined with the sample path continuity of Brownian motion $X_i$ implies that one can find a sequence of rationals $(r_n')$ such that  $r_n'\to s$, $X_i(r_n')\in B'$ and $X({r_n'})(B')=0$  for all $n$.
But this is impossible since by (\ref{eq:sup}), $X_i(r_n')\in S(X(r_n'))$
that implies  $X({r_n'})(B')>0$. The proof is complete.
\end{proof}

\begin{proof}[Proof of Corollary  \ref{co:suppor} ]
Combined with Lemma \ref{le:ut}, the desired modulus of continuity results for the support process in \ref{co:global}-\ref{co:localright} follow from Propositions \ref{le:upp}, \ref{le:localupp} and \ref{le:localupp2}.
Since the process $X$ is right continuous in $M_1(\mathbb{R}^d)$, applying \cite[Lemma 5.5]{LZ2}, we have $\mathbb{P}_\nu$-a.s.\,for any $t>0$,
$\rho_1(S(X(t)),S(X(t+\Delta t)))\to 0\,\, \text{as} \,\,\Delta t\to 0+.$
In addition, by \ref{co:global} we have $\mathbb{P}_\nu$-a.s.\,
$\rho_1(S(X(t+\Delta t)), S(X(t)))\to 0\,\, \text{as} \,\,\Delta t\to 0+.$
The desired result in \ref{co:righcon} then follows.
\end{proof}

\begin{proof}[{Proof of Corollary \ref{co:rho}}]
By Lemma \ref{th:localright2}, 	
one can show that given $c<\sqrt{{2}/(\beta-1)}$, \,  $\mathbb{P}_{\delta_0}$-a.s. for all small enough $t>0$,
$S(X(t))\not\subseteq \mathbb{B}(0, ch(t))$, which combined with Corollary \ref{co:suppor} \ref{co:localright} implies the desired result. The proof is complete.
\end{proof}





\end{document}